\IfFileExists{./prepreamble-amspreprint.sty}{\RequirePackage[packages,theorems,changes]{prepreamble-amspreprint}}{}

\documentclass[english]{amsart}

\RequirePackage{biblatex}
\RequirePackage{ltxcmds}
\IfFileExists{./preamble-amspreprint.sty}{\RequirePackage[packages,theorems,changes]{preamble-amspreprint}}{}

\usepackage{manuscript}
\usepackage{KMS-local}
\usepackage{scoop-packages}
\usepackage{scoop-semantic}
\usepackage{changes}
\usepackage{KMS-standard-packages}

\IfFileExists{./postpreamble-amspreprint.sty}{\RequirePackage[packages,theorems,changes]{postpreamble-amspreprint}}{}

\makeatletter
\@ifpackageloaded{changes}{
\definechangesauthor[name = {Fei Chen}, color = {red!80!black}]{FC}
\definechangesauthor[name = {Kirk M Soodhalter}, color = {blue!80!black}]{KMS}
}{}
\makeatother        

\makeatletter
\@ifpackageloaded{hyperref}{%
	\hypersetup{
		pdftitle = {Constraints on admissible behavior of \gmres applied to tridiagonal Toeplitz systems},
		pdfauthor = {Fei Chen, Kirk M Soodhalter},
		pdfkeywords = {Toeplitz matrices, GMRES, Krylov, Admissible convergence behavior},
		pdfcreator = {Created using the Scoop Template Engine version 1.6.0.post0+git.eeaa8072.dirty.}
	}
}{
	\pdfinfo{
		/Title (Constraints on admissible behavior of \gmres applied to tridiagonal Toeplitz systems)
		/Author (Fei Chen, Kirk M Soodhalter)
		/Subject ()
		/Keywords (Toeplitz matrices, GMRES, Krylov, Admissible convergence behavior)
		/Creator (Created using the Scoop Template Engine version 1.6.0.post0+git.eeaa8072.dirty.)
	}
}
\makeatother

\title[\gmres behavior for tridiagonal Toeplitz systems]{Constraints on admissible behavior of \gmres applied to tridiagonal Toeplitz systems}

\author[F. Chen]{Fei Chen}
\address[F. Chen]{School of Mathematics, Trinity College Dublin, College Green, Dublin 2, Ireland}
\email{\detokenize{chenf2@tcd.ie}}

\author[K. M. Soodhalter]{Kirk M Soodhalter}
\address[K. M. Soodhalter]{School of Mathematics, Trinity College Dublin, College Green, Dublin 2, Ireland}
\email{\detokenize{ksoodha@maths.tcd.ie}}
\urladdr{https://math.soodhalter.com}

\thanks{Both authors have been supported by Research Ireland grant 22/EPSRC/3857.}

\date{\today}

\dedicatory{}

\begin{document}

\begin{abstract}
The result of Greenbaum, Pták, and Strakoš that for a given set of eigenvalues, any
convergence curve is possible [SIMAX 1996] and the subsequent parameterization
of such matrix-right-hand-side pairs $(A,\bb)$ of Arioli, Pták, and Strakoš
[BIT 1998] demonstrated that the behavior of \gmres could not be completely
characterized by the eigenvalues of $A$ alone.  In this paper, we consider how
to use this theory to understand the admissible and attainable \gmres behavior
for matrices with constrained structure, focusing on non-Hermitian
(non-symmetric) tridiagonal Toeplitz matrices.  We show that Toepliz structure
necessarily constrains how the theory from these papers can manifest but
that a continuum of admissible behaviors is still attainable.  

\end{abstract}


\makeatletter
\ltx@ifpackageloaded{hyperref}{%
\subjclass[2010]{\href{https://mathscinet.ams.org/msc/msc2020.html?t=65F10}{65F10}, \href{https://mathscinet.ams.org/msc/msc2020.html?t=15B05}{15B05}}
}{%
}
\makeatother

\maketitle

\section{Introduction}
\label{section:introduction}
Discretizing \pdes leads to linear systems with large, sparse coefficient
matrices. When linear, constant-coefficient \pdes with Dirichlet boundary
conditions are discretized on uniform meshes, one can obtain Toeplitz,
multilevel Toeplitz and/or block Toeplitz systems
\cite{localtoepseq,localtoepseq2}. Toeplitz matrices have constant diagonals,
and multilevel and block Toeplitz matrices have related structures that can be
exploited to accelerate the convergence of iterative methods, and aid in
convergence analysis.  Such systems are widely solved by Krylov subspace
iterative methods, and we focus on their convergence behavior in this work. 

We consider the linear system 
\begin{equation}
\label{eqn:Axb}
A\bx
=
\bb
,
\quad
\mbox{with}
\quad 
A\in\C^{n\times n}
,
\quad
\mbox{and}
\quad
\bx, \bb 
\in 
\C^n
,
\end{equation}
where the matrix $A$ is non-symmetric tridiagonal and Toeplitz, $\bb$ is a
known right-hand side, and $\bx$ is the unknown solution. The majority of our
analysis is done over $\C$, to be as general as possible, but in
\Cref{subsection:admissible-attainable} we restrict our analysis to $\R$ for
simplicity. The Generalized Minimal Residual Method (\gmres)
\cite{SaadSchultz:1986:1} has long been a workhorse for solving such
large-scale problems arising in the computational sciences and engineering.
However, there remain gaps in our theoretical understanding of its convergence
behavior.

For symmetric systems, the spectrum and the right-hand side completely
determine the convergence behavior of a Krylov subspace method. However, it has
long been established that the residual convergence behavior of \gmres applied
to a general, non-symmetric linear system cannot fully be described by the
distribution of the eigenvalues of $A$; in extreme, pathological cases, the
eigenvalues bear no relation to the convergence rate
\cite{GreenbaumPtakStrakos:1996:1}. It is proven constructively that one can build a
matrix/right-hand-side pair $(A,\bb)$ with arbitrary eigenvalues exhibiting any
\gmres admissible convergence behavior. 

In this paper, we describe this construction in terms of constraints on the pair $(A,\bb)$, 
and we show that this construction is much more constrained when we
restrict $A$ to be \emph{non-symmetric}, \emph{tridiagonal}, and \emph{Toeplitz}.  Following the
theory developed in
\cite{GreenbaumPtakStrakos:1996:1,ArioliPtakStrakos:1998:1,Meurant:2012:1}, we
show that:
\begin{itemize}
	\item the structure of $A$ highly constrains the structure of the orthonormal residual Krylov basis;
     see, \Cref{definition.Krylov-basis};
	\item attainable convergence behavior of \gmres applied to these problems is restricted;
	\item the upper-triangular matrix $R$ relating convergence behavior to the eigenvalues via
		the relation \eqref{eqn:h-Rs-relation} is also constrained.
\end{itemize}

The rest of this paper is organized as follows.  In \Cref{section:background},
we give a brief description of \gmres to establish our language followed by
descriptions of the \gmres convergence specification framework and
previous analysis of \gmres applied to tridiagonal Toeplitz systems.  In
\Cref{section:gmres-toeplitz-parameterization}, we develop theory concerning
how the \gmres convergence specification framework is constrained in the case
that the matrix is tridiagonal and Toeplitz.  In
\Cref{section:numerical-demonstrations}, we show some small numerical
demonstrations illustrating the theory, and in \Cref{section:conclusions} we
concisely summarize the meaning of these results and lay out what future work
there is to pursue.

\section{Background}
\label{section:background}
The \gmres method often begins with an initial approximation, $\bx_0$, and we
solve the problem $A(\bx_0 + \bt) = \bb\iff A\bt = \br_0$, where $\br_0 =
A(\bx-\bx_0) = \bb - A\bx_0$.  Thus \Wlog we consider the case $\bx_0=\bnull$
when applying \gmres to \eqref{eqn:Axb}. At the $k$th iteration, \gmres
selects $\bx_k$, such that
\begin{align*}
	\bx_k\in \cK_k(A,\bb) 
	\coloneq 
	\Span\braces{\bb, A\bb,\cdots,A^{k-1}\bb}
	.
\end{align*}  
Although this basis is rarely used in computation
for numerical stability reasons, it is important for our presentation in 
\Cref{subsection:gmres-parameterization} we denote as the $k$th Krylov matrix 
\begin{align}
	K_k 
	= 
	\begin{bmatrix}
		\bb 
		& 
		A \bb 
		& 
		A^2 \bb 
		& 
		\cdots 
		& 
		A^{k-1} \bb
	\end{bmatrix}
	\in 
	\R^{n\times k}.
	\label{eqn:krylov-matrix}
\end{align}
\begin{definition}\label{definition.Krylov-basis}
	For $k=n$, we denote the columns of the Krylov matrix $\braces{\bb, A\bb, \ldots, A^{n-1}\bb}$ 
	as the \emph{Krylov basis}.
\end{definition}
The approximation $\bx_k \in \cK_k(A,\bb)$ is selected according to the minimum
residual criterion
\begin{align}
	\bx_k 
	= 
	\argmin_{
		x\in\cK_k(A,\bb)
	}
	\norm[auto]{
		\bb
		- 
		A \bx
	}
	.
	\label{eqn:GMRES-full-LS}
\end{align}
It is well established that \gmres can also be formulated as a residual
projection method, whereby $\bx_k$ is chosen such that 
\begin{align}
	\br_k 
	\coloneq 
	\bb 
	- 
	A\bx_k
	\perp 
	A\cK_k(A,\bb) 
	\iff 
	\br_k 
	= 
	\prn{
		I - \Phi_k
	}
	\bb
	,
	\label{eqn:gmres-resid-orth-constraint}
\end{align}
where $\Phi_k$ is the orthogonal projector onto $A\cK_k(A,\bb)$.

Standard implementation details can be found in \cite{SaadSchultz:1986:1} and are elaborated upon in, \eg, \cite[Chapter 6.5]{Saad2003Iterative}.
We give a brief synopsis of practical \gmres implementation to motivate the construction allowing for specification of \gmres convergence from \cite{GreenbaumPtakStrakos:1996:1}.
The Arnoldi iteration builds an orthonormal basis $\braces{\bv_i}_{i=1}^k$ for $\cK_k(A,\bb)$, often via a Gram-Schmidt type process.  The process is compactly encapsulated 
as the Arnoldi relation 
\begin{align}
	A V_k 
	= 
	V_{k+1} \underline{H_k} 
	= 
	V_k H_k + h_{k+1,k}v_{k+1}e_k^T
	,
	\label{eqn:arnoldi-relation}
\end{align}
with $V_k$ having the basis vectors as columns, and $\underline{H_k}\in\R^{(k+1)\times k}$ being upper Hessenberg with entries being the coefficients generated by 
the orthogonalization process.  The Hessenberg matrix $H_k\in\R^{k\times k}$ contains the first $k$ rows of $\underline{H_k}$.  Since $V_k$ has orthonormal columns, 
it follows from \eqref{eqn:arnoldi-relation} that $H_k = V_k^T A V_k$, meaning the Arnoldi iteration is a partial orthogonal Hessenberg factorization of $A$ generated by the starting 
vector $\bb$.  Generically, if there is no breakdown, the iteration can be carried to $k=n$ to produce a complete orthogonal upper Hessenberg factorization of $A$,
\begin{align}
	A 
	=
	V_n H_n V_n^T 
	\eqcolon 
	V H V^T,
	\label{eqn:full-Hessenberg-factorization}
\end{align}
where to simplify the notation, we have dropped the indices for the full factorization.  

\subsubsection*{Relevant theory of \gmres convergence}
Many authors have presented analyses illuminating the mechanics of \gmres convergence.  
We give a brief overview of work relevant to what we develop in \Cref{section:gmres-toeplitz-parameterization}.

In the years before the developments in \cite{GreenbaumPtakStrakos:1996:1},
much work was done to understand the bounds on attainable behavior of the
\gmres method with respect to spetral quantities.  
\begin{assumption}
	For our purposes, we mostly restrict ourselves to the case that $\bb$ 
	has components in all directions of the (generalized) eigenbasis,
	to avoid cases of early-terminating \gmres.  The one exception is 
	we do consider the case of early termination in fully developing 
	the results of \Cref{lemma:F-Jacobian-singular}.
\end{assumption}   
Worst-case behavior of \gmres has been explored extensively, \eg, in
\cite{ZavorinOLearyElman:2003:1,LiesenTichy:2004:1}.  Conversely, the authors
of \cite{GreenbaumTrefethen:1994:1} studied the ideal behavior of \gmres
for a given spectrum, and they developed a notion of a \gmres \emph{envelope},
in which the \gmres residual curves for general non-early-terminating \gmres
convergence curves must live.  For $4\times 4$ matrices, examples which do and
do not achieve the ideal are explicitly constructed \cite{Toh:1997:1}. The
well-known field of values bound of Elman encodes how the convergence of \gmres is
affected by non-normality \cite{EisenstatElmanSchultz:1983:1,Elman:1982:1}.

In \cite{Ipsen:2000:1}, the author developed bounds on the \gmres residual, in
particular for Toeplitz and Jordan blocks.  This work takes advantage of the
Toeplitz structure, and the author also characterizes how the decrease in
residual norm is connected to the ill-conditioning of the Krylov basis
\eqref{eqn:krylov-matrix}. Special cases, such as $\bb$ being a standard basis
vector, are also considered, simplifying some of the bounds. The author of
\cite[Theorem 1]{Sadok:2005:1} characterizes $\norm[auto]{\br_k}$ using a well-known fact
involving the Gram determinant, namely that $\norm[auto]{\br_k}$ can be  
characterized as the distance of $\bb$ to $A\cK_k(A,\bb)$, and this can be quantified 
as a ratio of two Gram determinants.  Similarly, this was also quantified 
in \cite[Equation 2.4]{EiermannErnst:2001:1} in terms of subspace angles. These give the two 
residual norm identities
\begin{align}
	\norm[auto]{\br_k}^2
	= 
	\frac{
		\det\prn{K_{k+1}^\ast K_{k+1}}
	}{
		\det\brackets{\prn{AK_k}^\ast AK_k}
	}
	= 
	\norm[auto]{\br_0}^2
	\sin^2\angle\prn{\br_0,A\cK_k\prn{A,\br_0}}
	.
	\label{eqn:gmres-gram-determinant-subspace-angle}
\end{align}
These give two definitions of direct mappings between the parameters
determining $(A,\bb)$ and the \gmres residual norm convergence history. 

The theory of \gmres convergence quantifies how various properties of $A$ and
$\bb$ determine the residual convergence behavior of \gmres. The language of
orthogonal Hessenberg factorizations allows us to more explicitly develop the
tools of \gmres behavior parameterization and specification.

\subsection{General parameterization of \gmres behavior}
\label{subsection:gmres-parameterization}
	We describe the theory presented in, \eg,
	\cite{ArioliPtakStrakos:1998:1,GreenbaumPtakStrakos:1996:1,GreenbaumStrakos:1994:1},
	that we leverage to gain further understanding of the behavior of
	\gmres for Toeplitz systems.  We note that Liesen has previously used
	this approach to obtain general \gmres convergence results
	\cite{Liesen:2000:1,Liesen:1998:1}.

	We restrict ourselves to the case that \gmres runs for the a full $n-1$
	iterations \footnote{The case of early-terminating \gmres has been
	characterized in \cite{DuintjerTebbensMeurant:2013:1}}. The Arnoldi
	iteration constructs a sequence of partial orthogonal Hessenberg
	factorizations of $A$.  Carried to the final step, it generically produces
	the full factorization \eqref{eqn:full-Hessenberg-factorization}.  Let
	$C\in\C^{n\times n}$ be the companion matrix associated to the
	characteristic polynomial of $A$, $p_A(z)=\det(zI - A)=z^n -
	\sum_{i=0}^{n-1} c_i z^i$, which has the same eigenvalues as $A$, with
	multiplicity, meaning $\deg p_A= n$. This matrix has the structure
	\begin{align*}
		C
		=
		\begin{bmatrix}
			0 & 0 & \dots & 0 & c_0 \\
			1 & 0 & \dots & 0 & c_1 \\
			0 & 1 & \dots & 0 & c_2 \\
			\vdots & \vdots & \ddots & \vdots & \vdots \\
			0 & 0 & \dots & 1 & c_{n-1}
		\end{bmatrix}
	\end{align*}
	and is also similar to $A$.  Extending the Krylov matrix
	\eqref{eqn:krylov-matrix} to the $k=n$ case (denoting $K\coloneq K_n$), we
	have the Krylov companion matrix relation 
	\begin{align}
		AK
		=
		KC
		\iff 
		A 
		=
		KCK^{-1}
		.
		\label{eqn:companion-matrix-relation}
	\end{align}
	We note that the full set of Arnoldi vectors is the unitary factor in a \qr-factorization of $K$; \ie,
	\begin{align*}
		K 
		= 
		V
		(DU)
	\end{align*}
	with the upper triangular factor being decomposed into a diagonal component $D$ and a unit-diagonal triangular component $U$. This
	allows us to express the coefficient matrix according to the decomposition
	\begin{align}
		A
		=
		V\underbrace{DUCU^{-1}D^{-1}}_{H}V^{\ast}.
		\label{eqn:A-full-gmres-decomp}
	\end{align}
	
	\subsubsection*{Degrees of freedom (\dof)}
	The proof of the existence of pairs $(A,\bb)$ for which the residual
	convergence of \gmres and the development of Ritz values can be
	completely unrelated to the eigenvalues is constructive. 
	Observe that
	over the complex numbers, there are $2n^2$ real \dof
	\footnote{
		by which we mean that for each entry of $A$, the real and
		imaginary parts represent two real \dof.
	}
	determining $A$ and $2n$ \dof determining $\bb$.  Thus, there are
	$2n^2+2n$ \dof we can use to construct pairs $(A,\bb)$ exhibiting
	prescribed \gmres residual convergence behavior and having prescribed
	eigenvalues.
	\begin{remark}
		\label{remark:real-vs-complex-dof}
		We note for the reader that as in the general theory, we pose this
		construction over $\C$ and count real \dof  If we restrict to
		constructing real matrices with real \dof, we still count any
		complex conjugate pairs as being specified by two real \dof 
		A similar count shows that a real matrix $A$ is
        	specified in total by $n^2$ \dof and $\bb$ by
		$n$ \dof When restricting to the real matrix case, we consider
		orthogonal transformations instead of unitary transformations.
	\end{remark} 
	
	We count the number of quantities (\eg eigenvalues, Ritz values,
	other spectral structures) we can assign in terms of the number of \dof each
	assignment fixes to understand what we can assign	independently.
	A key observation in works such as
	\cite{GreenbaumPtakStrakos:1996:1,TebbensMeurant:2012:1} is that we can
	work to a great extent directly with \eqref{eqn:A-full-gmres-decomp} to
	fix all the key pieces of information in $A$.

	\subsubsection*{Eigenvalues} 
	By assigning the $n$ coefficients in the characteristic polynomial in
	the last column of $C$, we assign the $n$ eigenvalues of $A$. This
	accounts for $2n$ real \dof In the general case, this can be done
	independently of any other quantities in the decomposition, and it is
	the main result of \cite{GreenbaumPtakStrakos:1996:1} that they can
	also be assigned fully independently of \gmres residual convergence. In
	the real case, eigenvalues account for $n$ \dof

	\subsubsection*{Ritz values for all iterations}
	Following \cite{TebbensMeurant:2012:1}, we consider the case of specifying the Ritz values at each iteration.
	For every iteration $j$, there are $j$ Ritz values.  Taken together,
	there are thus $n(n-1)/2$ Ritz values in total.  It has been shown that the strictly upper triangular entries  of each column of
	the unit-diagonal, upper-triangular matrix 
	\begin{align*}
		U^{-1}
		=
		\begin{bmatrix}
		    1 & c_0^{(1)} & c_0^{(2)} & c_0^{(3)} & \cdots & c_0^{(n-1)} \\
		    0 & 1 & c_1^{(2)} & c_1^{(3)} & \cdots & c_1^{(n-1)} \\
		    0 & 0 & 1 & c_2^{(3)} & \cdots & c_2^{(n-1)} \\
		    0 & 0 & 0 & 1 & \cdots & \vdots \\
			\vdots & \vdots & \vdots & \vdots & \ddots & c_{n-2}^{(n-1)} \\
		    0 & 0 & 0 & 0 & 0 & 1
		\end{bmatrix}
	\end{align*}
	are the coefficients of the characteristic polynomials of $H_j$ for each $j$ \cite{TebbensMeurant:2012:1}.  In other words, the characteristic polynomial of
	$H_j$ for $1\leq j < n$ is
	\begin{align*}
		p_j(z)
		=
		\det(zI - H_j)
		=
		z^j - \sum_{i=0}^{j-1} c_i^{(j)} z^i
		.
	\end{align*}
	This demonstrates that the $n(n-1)$ real \dof uniquely determining the
	Ritz values can be assigned by appropriately selecting the strictly
	upper-triangular entries of $U^{-1}$. In the real case, this becomes
	$n(n-1)/2$ \dof It was shown in \cite{TebbensMeurant:2012:1} that these
	can be assigned independently of the eigenvalues and the \gmres
	residual convergence, in most circumstances.  The one exception is that
	it is well documented \cite[Chapter 6.5]{Saad2003Iterative} that if \gmres stagnates at iteration $j$, then
	it must follow that at least one Ritz value is zero; and it is eigenvalue of $H_j$.

	\subsubsection*{Orthonormal basis for residual Krylov subspace}  
	The final degrees of freedom uniquely determining $A$ can be fixed
	through selection of a unitary (orthogonal) matrix with columns being an orthonormal
	basis of the Krylov subspace.  However, as established in
	\cite{GreenbaumStrakos:1994:1}, it is useful to consider fixing the
	orthonormal basis for the \emph{residual} Krylov subspace. As we only
	consider the case in which we construct $A$ to be non-singular, this is
	equivalent to specifying $V$. Let
	$\paren[auto]{\{}{\}}{\bw_1,\bw_2,\ldots,\bw_n}$ be an orthonormal basis
	for $\C^n$, which we take to be the columns of unitary $W\in\C^{n\times
	n}$ such that  
	\begin{align*}
		\mathrm{span}\paren[auto]{\{}{\}}{\bw_1,\bw_2,\ldots,\bw_j} = A\cK_j(A,\bb)
	\end{align*}
	for all $j\leq n$, and it follows that this basis can be obtained via
	the \qr-factorization
	\begin{align}
		AK 
		= 
		WR
		.
		\label{eqn:residual-Krylov-qr}
	\end{align}
	\begin{remark}\label{remark:WR-uniqueness}
		The APS-parameterization theory does not make any 
		assumptions on $R$ other than it is upper triangular.  However, 
		\Wlog we enforce uniqueness by insisting that $R$ have 
		real, positive diagonal entries corresponding to the norms computing 
		during a Gram-Schmidt-type orthogonalization process.  For an arbitrary 
		$WR$ arising in the APS parameterization, the rows and $R$ and the columns 
		of $W$ can be scaled such that we obtain a unique \qr-factorization with $W$ 
		successively spanning the same residual Krylov subspaces.
	\end{remark}
	\begin{remark}
		\label{remark:AK-WR-column-relationship}
		Equation \ref{eqn:residual-Krylov-qr} implies that
		$A^j\bb = \sum_{i=1}^j r_{ij}\bw_i$ where $R\eqcolon
		\prn{r_{k\ell}}$.
	\end{remark}
	
	Unitary matrices do not form a subspace in the vector space of matrices
	due to lack of closure under linear combination.  However, the number
	of real \dof can be counted by observing that the equations $W^\ast W =
	I$ constrain $n^2$ real \dof  (following
	\cite{EdelmanAriasSmith:1998:1})
	\footnote{
		That unitary matrices have $n^2$ real \dof arises from the fact
		that $W^\ast W=I$ implies  that $\bw_j^\ast\bw_i=0$, $i\neq j$
		which generates $n(n-1)/2$ complex constraints.  Each of these
		is  equivalent to solving two real equations fixing the real
		and imaginary parts of the scalar product.  Thus, they
		contribute $n(n-1)$ real constraints.  In addition, the
		equations $\bw_i^\ast\bw_i=1$ contribute $n$ additional real
		equations. 
	}
	. 
	\begin{remark}
		\label{remark:unitary-scaling-W-overcount}
		This overcounts the \dof, since from
		\Cref{remark:AK-WR-column-relationship}, the successive span of
		an orthonormal basis is invariant under unitary scaling of the
		basis vectors.  Thus, for each column $\bw_k$ of $W$, we must
		account for an extra \dof from $\bw_k\mapsto
		e^{i\theta_k}\bw_k$, $\theta_k\in\R$.  Up to unitary scaling of
		the columns, there are $n^2-n$ real \dof determining
		$W$.	
	\end{remark} 
	Similarly, for real, orthogonal matrices, the equations $W^TW=I$
	imposes $n(n+1)/2$ constraints, leaving $n(n-1)/2$ \dof In
	\cite{GreenbaumPtakStrakos:1996:1}, it was shown these can be assigned
	independently of all other quantities. 
	
	Altogether, this determines complex $A$ uniquely. Counting up the number
	of quantities we can assign verifies that all \dof are accounted for since $2n + n(n-1) +
	n^2 -n = 2n^2$ real \dof For real $A$, we similarly verify $n + 2\cdot
	n(n-1)/2 = n^2$ \dof

	\subsubsection*{Specification of convergence behavior}
	Due to its residual minimization over nested subspaces of increasing
	dimension, \gmres must produce a monotonically non-increasing sequence
	of residuals.  
	\begin{definition}
		\label{definition:admissible-convergence}
		A sequence $\braces{f_i}_{i=0}^{n-1}\subset\R^+$ is called an \emph{admissible \gmres convergence sequence} if it is monotonically non-increasing; \ie
		\begin{align*}
			f_0 \geq f_1 \geq f_2 \geq \cdots \geq f_{n-1} > 0 
			.
		\end{align*}
	\end{definition}
	The specification of an admissible \gmres residual convergence (in the
	sense of \Cref{definition:admissible-convergence}) is used to determine
	an appropriate $\bb$.  It has been observed that we can write $\bb =
	VDU\be_1$. However, this does not yield a clear avenue
	for specifying a wanted convergence pattern.  Instead, we lean on the
	residual orthogonality constraint
	\eqref{eqn:gmres-resid-orth-constraint}, using the basis $W$.  
	
	  If we express the right-hand side with respect to this basis $\bb=\sum_{i=1}^n \eta_{i}\bw_i$, then we observe that
	the residual orthogonality constraint \eqref{eqn:gmres-resid-orth-constraint} implies that $\br_j=\sum_{i=j+1}^n \eta_{i}\bw_i$.
	Exploiting basic properties of norms and orthonormal bases, it is observed in \cite{GreenbaumPtakStrakos:1996:1} that we can thus
	write $|\eta_{i}|^2 = \norm[auto]{\br_{i-1}}^2 - \norm[auto]{\br_i}^2$, for each $i$.  
	Thus we can construct $\bb$ in this basis by determining the $n$
	coefficients $\curly{\eta_{i}}_{i=1}^n$ (\ie $2n$ real \dof), and we can choose them to have absolute
	values that allow us to implicitly specify a right-hand side $\bb$
	such that \gmres will exhibit any wanted admissible residual
	convergence pattern for the matrix $A$.  

	It is shown in \cite[Theorem 2.1]{ArioliPtakStrakos:1998:1} how to construct a
	parameterization of any nonsingular matrix $A$ and how to construct a
	linear system with prescribed spectrum and convergence behavior.  They
	specify $A$ using the equivalent \emph{APS parameterization}, using $W$
	and $R$ from \eqref{eqn:residual-Krylov-qr}.  We restate the theorem
	using this parameterization.

	\begin{theorem}{(\cite[Theorem 2.1]{ArioliPtakStrakos:1998:1})}
	\label{thm:APS-parameterization}
	Let $\braces{f_i}_{i=0}^{n-1}$ be an admissible convergence sequence as
	defined in \Cref{definition:admissible-convergence} and let
	$\lambda_1,\cdots, \lambda_n$ be $n$ non-zero complex numbers. Let
	$A\in\C^{n\times n}$, $\bb\in\C^n$, and $\bx_0=\bnull$. Then the following
	assertions are equivalent:
	\begin{enumerate}
		\item The spectrum of $A$ is $\{\lambda_1,\cdots, \lambda_n\}$
			and \gmres applied to the pair $\prn{A,\bb}$ yields
			residuals $\bb=\br_0,\cdots, \br_{n-1}$ such that
			\begin{align*}
				\norm[auto]{\br_j}
				=
				f_j,
				\quad 
				j
				=
				0,1,\cdots,n-1
				.
			\end{align*}
		\item The matrix $A$ can be expressed in terms of the residual Krylov basis 
			\begin{align}
				A
				=
				WRCR^{-1}W^*
				\label{eqn:APS-param}
			\end{align}
			and $\bb=W\bh$, where $C$ is the companion matrix
			corresponding to the characteristic polynomial of $A$,
			$W$ is unitary, and $R$ nonsingular upper triangular
			such that
			\begin{align}
				R\bs
				=
				\bh
				.
				\label{eqn:h-Rs-relation}
			\end{align}
	\end{enumerate}
        
	The vectors $\bs$, $\bh$ satisfy
	\begin{align*}
		\bs
		=
		\begin{bmatrix}
			\xi_1 & \cdots &\xi_n
		\end{bmatrix}^T
		,
		\quad
		\mbox{where} 
		\quad
		1-(\xi_1z+\cdots+\xi_nz^n)
		=
		\prod\limits_{i=1}^n
		\prn{
			1-\cfrac{z}{\lambda_i}
		}
		,
		\mbox{and}
		\\
		\bh
		=
		\begin{bmatrix}
			\eta_1
			&
			\cdots
			&
			\eta_n
		\end{bmatrix}^T
		\quad
		\mbox{where}
		\quad
		\abs[auto]{\eta_j}^2
		=
		f_{j-1}^2-f_j^2
		\quad 
		\mbox{for}
		\quad
		j=1,\cdots, n
		,
		\,
		f_n\equiv0
		.
	\end{align*}
\end{theorem}
This forms the basis for the Arioli-Pták-Strakoš (APS) parameterization
(described in \cite{ArioliPtakStrakos:1998:1} and further elaborated upon in
\cite{Meurant:2012:1}) on which we build our theory; see,
\Cref{thm:tridiag-toeplitz-APS}. It is important to note that from its
definition, $\bs$ is completely and uniquely determined by the eigenvalues, and
$R$ maps it to a vector $\bh$ encoding convergence.  What the authors of
\cite{ArioliPtakStrakos:1998:1} demonstrated is that the theory of
\cite{GreenbaumPtakStrakos:1996:1} allows for the construction of pairs
$(A,\bb)$ with arbitrary eigenvalues such that the associated $\bs$ (\ie, any
set of eigenvalues) is mapped by $R$ to any $\bh$ (\ie convergence curve). The
APS parameterization \cite{ArioliPtakStrakos:1998:1} describes, in particular,
how all matrices with a prescribed spectrum (via $\bs$ ) that exhibit (for some
$\bb$) a certain convergence behavior (via $\bh$) can be parameterized by all
upper-triangular matrices $R$ such that $R\bs=\bh$.  Given $\bs$ and $\bh$, the
full parameterization of all $(A,\bb)$ for which $A$ has the prescribed
spectrum and \gmres exhibits the prescribed convergence is parameterized by the
choice of $W$ and admissible $R$.  Conversely, for a specified $A$ with
spectrum described by $\bs$, the breadth of admissible $R$ and $\bh$ (in
essence telling us what mappings between eigenvalues and \gmres convergence
curves are possible). In a qualitative sense, one can consider this as a
roundabout measure of non-normality. We show in
\Cref{section:gmres-toeplitz-parameterization} that structural constraints
imposed on $A$ (in this case it being Toeplitz) constrain which unitary $W$
are admissible.
\begin{remark}\label{remark:joint-A-b-construction}
	It is imperative to remember that this construction of $(A,\bb)$ is
	joint.  The choice of $\bb$ is connected to the choices of $W$ and $R$
	and vice-versa.  In particular, the construction requires that we find
	an $R$ mapping $\bs$ to $\bh$ that specifies the wanted convergence.
	Furthermore, it is observed in, \eg \cite{ArioliPtakStrakos:1998:1},
	that $R$ contains the specified \gmres convergence behavior in its
	entries, again highlighting that $A$ and $\bb$ are constructed jointly.
\end{remark}

We observe that since we assume no early termination of \gmres, $K$ is
nonsingular. Since $\det(AK)=\det(A)\det(K)\neq 0$, the $\qr-$factorization of
$AK$ is unique up to unitary scaling of the rows of $R$ (or columns of $W$), as discussed in
\Cref{remark:unitary-scaling-W-overcount}. From $\bb=W\bh$ in the second
assertion, we have $\bh=W^*\bb$.  We note that there are a total of $2n^2+2n$
real \dof when independently specifying $A$ and $\bb$, but the theory in
\cite{ArioliPtakStrakos:1998:1,GreenbaumPtakStrakos:1996:1,Meurant:2012:1}
demonstrates that to get the specified \gmres residual convergence curve with
the specified eigenvalues, the matrix and right-hand side must be specified together, 
since $R$ and $\bh$ are tied together by the need for $R$ to map $\bs\mapsto\bh$.

\begin{remark}\label{remark:alternate-h-giving-same-convergence}
	Note that since $\bh$ determines the convergence curve via the absolute 
	value of its entries, for any other $\btilb$ for which \gmres produces the 
	same convergence curve must via the APS-parameterization be associated to 
	some
	\begin{math}
		\btilh 
		= 
		\begin{bmatrix}
			s_1 \eta_1 & s_2 \eta_2 & \cdots & s_n\eta_n
		\end{bmatrix}
	\end{math}
	wherein $s_j=e^{i\pi\theta_j}$
    	have modulus $1$ (\ie complex signs). However,
	it is not simply a matter of constructing $\btilh$ and expanding in the $W$
	basis.  As pointed out in \Cref{remark:joint-A-b-construction}, $W$ and $R$ are
	bound to $\bb$.  Therefore, for a set of roots of unity $\braces{s_i}_{i=1}^n$,
	any $\btilb$ associated to an $\btilh$ (and an
	associated unitary $\tilW(\btilb)$) exhibiting the same \gmres convergence
	pattern as for $\bb$ with the same $A$ is the solution of the constrained non-linear
	equation 
	\begin{align*} 
		\btilb 
		= 
		\sum_{i=1}^n 
		s_i 
		\abs{\eta_{i}}
		\btilw_i(\btilb) 
		\  
		\mbox{s. t.}
		\  
		\tilW\tilR 
		= 
		\begin{bmatrix}
			A\btilb & A^2\btilb & \cdots & A^n\btilb
		\end{bmatrix}
		\ 
		\mbox{is a \qr factorization.}
	\end{align*}
	This is non-trivial to solve, but we allude to an alternative approach in \Cref{section:numerical-demonstrations}.
\end{remark}


We show in \Cref{section:gmres-toeplitz-parameterization} the construction of
tridiagonal Toeplitz matrices with prescribed \gmres convergence according to
\Cref{thm:APS-parameterization} is highly constrained due to the matrix
structure and restricted \dof available.  The results concerning the APS
parameterization tell us that the relationship between eigenvalues and
convergence (\ie, $h=Rs$) is constrained by a restriction on admissible $R$.  
\begin{definition}
	\label{definition:admissible-W}
	From \eqref{eqn:residual-Krylov-qr}, we see that the unitary $W$ has orthonormal columns
	which successively span the same spaces as the columns of $AK$.  For a particular 
	matrix structure (\eg, Toeplitz), we call a unitary matrix that can successively span 
	the same space as $AK$ for some $A$ with that structure and some $\bb$ 
	\underline{admissible} with respect to that structure.
\end{definition}
We observe that this construction can be understood as specifying a matrix via its \emph{Frobenius normal form} in \Cref{section:frobenius-normal-form}.

\subsection{Theory of \gmres convergence for Toeplitz systems}
\label{subsection:gmres-toeplitz-conv-theory}
A general, non-symmetric $n\times n$ Toeplitz matrix is defined as having constant entries on each diagonal; \eg 
\begin{align*}
	T_n 
	= 
	\begin{bmatrix}
		a_0 & a_{-1}   & a_{-2} & \cdots & \cdots & a_{-(n-1)} \\
		a_1 & a_0      & a_{-1} & \ddots &        & \vdots \\
		a_2 & a_1      & \ddots & \ddots & \ddots & \vdots \\ 
 		\vdots & \ddots & \ddots & \ddots & a_{-1} & a_{-2} \\
 		\vdots &        & \ddots & a_1    & a_0    & a_{-1} \\
		a_{n-1} & \cdots & \cdots & a_2    & a_1    & a_0	
	\end{bmatrix}
	,
\end{align*}
meaning it is defined by its $2n-1$ unique entries.  This means one has this
many \dof when specifying an arbitrary Toeplitz matrix. These matrices are
highly structured, and there is a deep, interesting theory concerning their
spectra and how it develops as $n$ gets larger.  It is beyond the scope of this
work, but this leads to the robust theory of Toeplitz symbols and how they can
be used to estimate spectral properties of $T_n$
and for 
\emph{generalized locally Toeplitz} sequences; see, \eg,
\cite{localtoepseq,localtoepseq2} for a thorough treatment of this beautiful
theory.

We observe that one consequence of this Toeplitz structure is that the matrix is symmetrizable 
via left- or right-multiplication by the flip matrix (also known as the anti-identity matrix)
\begin{align*}
	Y 
	= 
	\begin{bmatrix}
		0&0&\cdots&0&1 
		\\ 
		0&0&\cdots&1&0 
		\\ 
		\vdots&0&1&0&0 
		\\ 
		0&\text{\reflectbox{$\ddots$}}&\vdots&\vdots&\vdots 
		\\ 
		1&0&\cdots&0&0 
	\end{bmatrix}
	.
\end{align*}
The resulting matrix $YA$ (or $AY$) is symmetric and possibly indefinite, with
the absolute values of the eigenvalues of $YA$ (or $AY$) being the singular
values of $A$; see, \eg, \cite{PestanaWathen:2015:1}.  This leads to the
innovation that non-symmetric Toeplitz systems can be solved via flip
preconditioning using \minres
\cite{BarbarinoEkstroemGaroniMeadonSerraCapizzanoVassalos:2023:1,Pestana:2019:1,PestanaWathen:2015:1}.
This symmetrization property has also been extended in a constructive sense to
block- and multi-level Toeplitz matrices
\cite{Pestana:2019:1,PestanaWathen:2015:1}.  More abstractly, Wathen has
recently explored  the fact that every non-symmetric matrix is symmetric with
respect to a collection of bilinear forms, and the nature of those bilinear
forms for a given matrix has many interesting consequences for the structure of
the spectrum \cite{Wathen:2025:1}.

The resulting preconditioned \minres residual convergence behavior for the matrix $HA$ (or $AH$) is in 
large part dictated by the eigenvalues of the system matrix (\ie $\pm$ the singular values of $A$).  In the 
context of the theory presented in \Cref{subsection:gmres-parameterization}, one can hypothesize that the 
limited \dof determining a Toeplitz matrix also constrains in some way the range of 
attainable behaviors of \gmres applied to a Toeplitz system.  

\subsubsection*{Existing theory}
We discuss existing theory of \gmres convergence for Toeplitz systems, focusing
on aspects that motivate this work.  The study of Toeplitz matrices has long
been of general mathematical interest, and \gmres convergence for
Toeplitz systems (particularly for large, sparse Toeplitz systems and low-rank
perturbations thereof) is of interest since many systems arising from the
discretization of \pdes have this structure.  

In this work, we restrict ourselves to the most constrained case: a
non-symmetric tridiagonal Toeplitz matrix, which only has three \dof, with the
matrix being determined by a triplet $(a,b,c)$; \ie, 
\begin{align}
	A
	=
	\begin{bmatrix}
		a & c & 0 & \cdots & 0
		\\
		b & a & c & \cdots & 0
		\\
		\vdots & \ddots & \ddots & \ddots & \vdots
		\\
		0 & 0 & b & a & c 
		\\
		0 & 0 & 0 & b & a
	\end{bmatrix}
	\eqcolon 
	\tritoep{a}{b}{c}
	,
	\label{eqn:tridiag-Toepl-gen}
\end{align}
	wherein we denote a tridiagonal matrix by its three entries
	$\tritoep{a}{b}{c}$, with the ordering: diagonal, subdiagonal,
	superdiagonal.
Such a matrix admits the splitting into the $3$-term sum 
\begin{align}
	A 
	= 
	a I 
	+ 
	c S 
	+ 
	b S^T
	\label{eqn:trdiag-Toepl-sum}
\end{align}
where $S\in\R^{n\times n}$ is the upward shift matrix with zero entries except
on the first super-diagonal, which has constant values of $1$.  Following
\cite{LiesenStrakos:2004:1}, we exploit \eqref{eqn:trdiag-Toepl-sum} in our
analysis. In the simplest case, tridiagonal Toeplitz systems arise from the
discretization of certain one-dimensional \pdes with appropriate boundary
conditions. These systems are often overly simplified, but studying \gmres
behavior in this setting yields insights for more realistic and complicated
structured problems. 

It is well established (see, \eg, \cite[Pages 86 and 113]{Smith:1985:1}) that
for the case $bc\neq 0$ the eigenvalues and eigenvectors of the tridiagonal
Toeplitz matrix admit an expression in terms of the triplet $(a,b,c)$, namely
the eigenpair $(\lambda_i, \bz_i)$ can be written as 
\begin{align}
	\lambda_i 
	= 
	a + 2\sqrt{bc}\cos\prn{\frac{i\pi}{n+1}}
	\in\C
	\quad 
	\mbox{and}
	\quad
	\bz_i 
	= 
	\prn{
		(b/c)^{j/2}\sin\prn{\frac{ij\pi}{n+1}}
	}_j
	\in\C^{n}
	.
	\label{eqn:Toeptliz-eig}
\end{align}  
We observe that two of three \dof are committed to determining the eigenvalues.
The diagonal value $a$ is a point around which the eigenvalues are centered,
and the product $bc$ determines the nature and tightness of any clustering. We
note that it has been well established in, \eg
\cite{FrommerGlaessner:1998:1,Simoncini:2003:1}, that Krylov subspaces are
invariant with respect to a constant shift of the diagonal; \ie in this setting 
\begin{align}
	\cK(A, \bb) 
	= 
	\cK(cS + bS^T, \bb)
	.
\end{align}
The analysis in \Cref{section:gmres-toeplitz-parameterization} establishes
exactly what role the diagonal element $a$ plays in prescribing \gmres behavior for
tridiagonal Toeplitz matrices via the APS-factorization. 

The eigenvectors are determined by a third \dof, the ratio $b/c$. Thus, we have
limited ability to specify spectral structure if we want to obtain a
tridiagonal Toeplitz matrix. If we specify eigenvalues conforming to
\eqref{eqn:Toeptliz-eig}, this leaves only one \dof to specify much else when
building pairs $(A,\bb)$ exhibiting a given \gmres residual convergence
behavior.  It should be noted that even in this constrained scenario, it has
been shown that one can construct tridiagonal Toeplitz cases for which the
spectrum is misleading, when predicting \gmres convergence. Indeed, we show in
\Cref{section:gmres-toeplitz-parameterization} that the ratio $b/c$ is
linked to the characterization of admissible APS-parameterization, \ie $W$
satisfying \Cref{definition:admissible-W}.

\begin{assumption}
	We observe that \eqref{eqn:Toeptliz-eig} holding for $bc\neq 0$
	excludes the case of $A$ being a scaled Jordan block or the transpose
	thereof. We proceed under the assumption that $bc\neq 0$ and address
	$bc=0$ in \Cref{subsection:dealing-with-jordan-blocks}.
	\label{assumption:bc-neq-0}
\end{assumption}

Liesen and Strakoš studied the tridiagonal case in
\cite{LiesenStrakos:2004:1}, building up from theory on bidiagonal and Jordan
block structure to the more general tridiagonal Toeplitz case.  They observe
that spectral analysis can be hampered by ill-conditioned eigenvector, a common
hindrance in the \gmres analysis setting.  The same authors also give a concise 
description of the sort of discretizations leading to the studied Toeplitz structure 
in \cite{LiesenStrakos:2005:1}.

This case has also been studied in the \phd thesis of Zhang
\cite{Zhang:2007:1}, and Li also investigated convergence of \cg and \gmres for
these systems \cite{Li:2007:1}. Li and Zhang subsequently published a sequence
of two papers building up additional theory of \gmres convergence for
tridiagonal Toeplitz systems \cite{LiZhang:2008:1,LiZhang:2009:1}.  In the
first paper, they produce bounds related to the three parameters defining
the matrix using Chebyshev polynomials, and in the second paper, they
demonstrate that the residual convergence behavior of \gmres applied to a given
tridiagonal Toeplitz system accumulates near the residual curve
produced when \gmres is applied for the same matrix but with right-hand side
being either the first or the last standard basis vectors $\be_1$ and $\be_n$. 

The behavior of \gmres for these starting vectors in the
degenerate case of a Jordan block has been explored \cite{Ipsen:2000:1,IpsenMeyer:1998:1}.
Indeed, straightforward computations have shown (\eg, in
\cite{IpsenMeyer:1998:1}) that the Arnoldi process generates the columns of the
identity, in increasing and decreasing order. When applied to $A=\tritoep{a}{b}{c}$ 
with starting vectors $\be_1$ and $\be_n$, respectively.  When this is carried to the final
step, we find that $A$ is similar, respectively, to 
\footnotesize
\begin{align*}
	H_{\be_1} 
	=
	\begin{bmatrix}
		a && \sigma_b c & & & 
		\\
		|b| & a && \sigma_b c & & 
		\\
		& \ddots & \ddots & \ddots & 
		\\
		& & |b| & a && \sigma_b c 
		\\
		& & & |b| & a
	\end{bmatrix}
	\quad 
	\mbox{and}
	\quad 
	H _{\be_n}
	=
	\begin{bmatrix}
		a && \sigma_c b & & & 
		\\
		|c| & a && \sigma_c b & & 
		\\
		& \ddots & \ddots & \ddots & 
		\\
		& & |c| & a && \sigma_c b 
		\\
		& & & |c| & a
	\end{bmatrix}
	,
\end{align*}
\normalsize
where $\sigma_b\coloneq\sign\prn{b}$ and $\sigma_c\coloneq\sign\prn{c}$,
thereby demonstrating a trivial dependence of the Hessenberg matrices on $(a,b,c)$
for these particular right-hand sides.

\section{Parameterization of \gmres behavior for Toeplitz systems}
\label{section:gmres-toeplitz-parameterization}

Our approach for studying admissible convergence behavior for \gmres applied to
tridiagonal Toeplitz problems is to study how the structure of these matrices
and limited \dof constrains attainable APS parameterizations of the form
\eqref{eqn:APS-param}, using \eg the relationship described in
\Cref{remark:AK-WR-column-relationship}.  

We consider the general question: 
\emph{
	for an arbitrary unitary (or orthogonal in the real case) matrix $W$, does
	there exist $A=\tritoep{a}{b}{c}$ and upper-triangular
	matrix $R$ (with positive diagonal) such that, for conforming companion
	matrix $C$, $A$ can be written in the form \eqref{eqn:APS-param}?
}  
Given the limited degrees of freedom, the answer is most certainly \textbf{no} for $n\geq 4$, and we
develop theory describing how the structure of $W$ (and indeed $R$) must be constrained. 

Considering the negative answer to the general question, we make a second, more targeted inquiry: 
\emph{
	if $W$ is a unitary (orthogonal) matrix arising from the APS
	parameterization \eqref{eqn:APS-param} generated by a known pair
	$(A,\bb)$, $A=\tritoep{a}{b}{c}$, does there exist another pair $(\wtA,
	\btilb)$ with $\wtA=\tritoep{\tila}{\tilb}{\tilc}$ that generates the
	same $W$?  Can such a pair be found generating the same \gmres residual
	convergence behavior?
}

We lay out some general results that are useful for answering these questions,
focusing first on characterizing admissible unitary matrices $W$ and upper
triangular matrices $R$ satisfying \eqref{eqn:APS-param} and then building up
to specifying convergence behavior. Using the APS-parameterization
\eqref{eqn:APS-param}, we illustrate the constrained nature of \gmres
convergence by considering different cases in which one may try to obtain a
pair $(A,\bb)$ exhibiting prescribed \gmres residual convergence behavior for
which $A$ is tridiagonal and Toeplitz. Namely, we follow
\cite{ArioliPtakStrakos:1998:1} and rewrite \eqref{eqn:APS-param} as 
\begin{align}
	\widehat{A}
	\coloneq
	W^\ast A W 
	= 
	R C R^{-1}
	,
	\label{eqn:Hessenbergization-residual-Krylov}
\end{align}
which leads to the relation 
\begin{align}
	\widehat{A}R 
	= 
	RC
	.
	\label{eqn:B-R-Krylov-relation}
\end{align}
Since $R$ is upper triangular, it follows (adopting \matlab notation) that 
\begin{align}
	R(:,i) 
	= 
	\sum_{j=1}^i 
	r_{ji}\be_j
	, 
	\label{eqn:triangular-structure-R}
\end{align}
where $\be_j$ denotes the $j$th standard basis vector, which returns the $j$th
column of a matrix when multiplied on the right. We note that
\eqref{eqn:B-R-Krylov-relation} is a Krylov-type relation (following similar
logic to a derivation used in \cite{Meurant:2019:1}), since 
\begin{align}
	\widehat{A} R(:,i) 
	=& 
	R(:,i+1) 
	\quad 
	\mbox{for}
	\quad 
	1\leq i \leq n-1
	\nonumber
	\\
	\widehat{A} R(:,n) 
	=&
	\sum_{j=0}^{n-1} 
	c_j R(:,j+1)
	,
	\label{eqn:B-action-on-R}
\end{align}
due to the companion structure of $C$. 
\begin{remark}
	\label{remark:Ahat-always-upper-Hessenberg}
	Note that for $n\geq 3$, \eqref{eqn:Hessenbergization-residual-Krylov} implies that $\widehat{A}$ is upper Hessenberg.
	In \cite{ArioliPtakStrakos:1998:1}, the authors transform away $W$ in order to understand the common Hessenberg structure 
	determined by $R$ and $C$ for the APS-parameterization.
\end{remark}

From \eqref{eqn:trdiag-Toepl-sum} along with the fact that $W$ is 
unitary, we observe that $\widehat{A}$ has the structure 
\begin{align}
	\widehat{A} 
	= 
	a\cdot I 
	+ 
	c
	\cdot
	W^\ast S W 
	+ 
	b 
	\cdot
	W^\ast S^T W
	\label{eqn:B-Toepl-sum}
	.
\end{align}
 
\begin{lemma}
	\label{lemma:shifted-Krylov-R}
	Let us denote $F_{b,c}\coloneq bS^T + cS$ and $\FbcHat \coloneq W^\ast
	F_{b,c} W$. It follows that the first $j$ columns
	of $R$ span $\cK_j\prn{\FbcHat, \be_1}, $for $j=1,2,\ldots, n,$.
\end{lemma}
\begin{proof}
	This follows from the Krylov-type relationship that $\widehat{A}$ maps the
	$R(:,i)$ to $R(:,i+1)$ for $i=1,2,\ldots, n-1$.  Since $\widehat{A} =
	\FbcHat + aI$, the identity shift invariance of Krylov
	subspaces \cite{FrommerGlaessner:1998:1,Simoncini:2003:1} yields the
	result.
\end{proof}

We re-express the Krylov relation \eqref{eqn:B-action-on-R}
\begin{align}
	\prn{\FbcHat + a I}
	R(:,i) 
	= 
	R(:,i+1)
	\quad 
	\mbox{for}
	\quad 
	i=1,2,\ldots, n-1
	.
	\label{eqn:general-B-W-R-constraints}
\end{align}
We see from \eqref{eqn:B-Toepl-sum} that $a$ does not interact with the residual Krylov basis. 
Using
\labelcref{eqn:B-action-on-R,eqn:B-R-Krylov-relation,eqn:B-Toepl-sum,eqn:general-B-W-R-constraints},
we make general conclusions about how the tridiagonal Toeplitz structure
\eqref{eqn:tridiag-Toepl-gen} interacts with the APS parameterization
\eqref{eqn:APS-param}. From this, we study how these general conclusions
manifest when applied for different dimensions $n$.

The Hessenberg structure discussed in \Cref{remark:Ahat-always-upper-Hessenberg} leads to the following.
\begin{corollary}
	\label{corollary:bc-complex-line-ratio}
	For a unitary matrix $W$ to be admissible for forming an APS factorization \eqref{eqn:APS-param} for $\tritoep{a}{b}{c}$, when $n \geq 3$,
	for each column $i=1,2,\ldots, n-1$, it must satisfy the homogeneous equations
    \begin{align} 
		\frac{b}{c} 
		= 
		- 
		\frac{
			\bw_j^\ast S \bw_i
		}{
			\bw_j^\ast S^T \bw_i
		}
		\quad 
		\mbox{for}
		\quad 
		i+2\leq j\leq n
		\label{eqn:bc-first-ratio-constraints}
		.
	\end{align}
\end{corollary}
\begin{proof}
	We observe that the upper Hessenberg structure implies that 
	\begin{align*}
		\FbcHat(i+2:n,i)
		= 
		\bnull
		.
	\end{align*}
	The result is obtained directly by solving each homogeneous equation associated to a zero entry of $\FbcHat$ for $b/c$.
\end{proof}
For all $n > 2$, it follows from \eqref{eqn:bc-first-ratio-constraints} that a
unitary $W$ arising from an APS parameterization of $A$ defines a 
complex line, passing through the origin, on which
$(-\bw_j^\ast S \bw_i,\bw_j^\ast S^T \bw_i)$ must lie. 
Further homogeneous equations yield other ratios of the form \eqref{eqn:bc-first-ratio-constraints} equal to $b/c$, 
which means they become constraints on $W$ itself.

\subsection{Does there exist tridiagonal Toeplitz $A$ for a given $W$?}
We consider the case of having an arbitrary unitary (or real \emph{orthogonal})
matrix $W$. The upper triangular matrix $R$ and $(a,b,c)$ are treated as
unknowns, except for $r_{11}$. 
\begin{remark}
	\label{remark:r11-always-free}
	Since \gmres behavior is invariant \wrt the scale of the right-hand
	side, we are always free to scale $\norm[auto]{\bb}$, and therefore
	$\norm[auto]{A\bb}$.  Thus, $r_{11}$ always remains free.
\end{remark}
 Can we reconstruct a tridiagonal Toeplitz matrix $A$ of the form
 \eqref{eqn:APS-param}? We build up the theory by treating the distinct cases
 $n=2,3,4$ and $n>4$.

\subsubsection*{The case $n=2$}
The $2\times 2$ case serves as the base case, but there is no interesting Krylov structure
to exploit.  There is only one possible Toeplitz matrix structure in this case.
There are no homogeneous equations imposing constraints of the form \eqref{eqn:bc-first-ratio-constraints}. 
The relevant equations are
\begin{align}
	r_{11}a 
	+
	r_{11}\FbcHat(1,1)
	=
	r_{12}
	\iff&
	a
	+
	b
	\prn{ 
		\bw_1^\ast S^T \bw_1 
	}  
	+ 
	c 
	\prn{
		\bw_1^\ast S \bw_1
	}
	=
	\frac{r_{12}}{r_{11}}
	\label{eqn:bc-fix-a-r12}
	\\ 
	r_{11}\FbcHat(2,1)
	=
	r_{22}
	\iff&
	b
	\prn{ 
		\bw_2^\ast S^T \bw_1 
	}
	+
	c 
	\prn{
		\bw_2^\ast S \bw_1}
	=
	\frac{r_{22}}{r_{11}}
	\label{eqn:b-c-complex-line}
	.
\end{align}
Equation \ref{eqn:bc-fix-a-r12} tells us that for a given
$(b,c)$ we are free to choose either $a$, which fixes $r_{12}$, or $r_{12}$,
which fixes $a$.  Equation \ref{eqn:b-c-complex-line} fixes $r_{22}$.

Thus, we have 
\begin{align*}
	A 
	= 
	\begin{bmatrix}
		a & c 
		\\ 
		b & a
	\end{bmatrix}
	,
	\quad
	R 
	=
	\begin{bmatrix}
		r_{11} & r_{12}
		\\
		       & r_{22}
	\end{bmatrix}
	,
	\quad 
	\mbox{and} 
	\   
	\begin{bmatrix}
		A\bb 
		& 
		A^2 \bb	
	\end{bmatrix}
	= 
	WR 
	. 
\end{align*}
The only equations that must be satisfied are \labelcref{eqn:b-c-complex-line,eqn:bc-fix-a-r12}.  From this we can prove the following. 
\begin{lemma}
	\label{lemma:2x2-any-unitary-APS}
	Let $W\in\C^{2\times 2}$ be unitary.  Then there exist infinitely many
	Toeplitz matrices $A$ and upper triangular matrices $R$ with positive
	diagonal entries such that $A$ can be written as an APS
	parameterization product of the form \eqref{eqn:APS-param}.
\end{lemma}
\begin{proof}
	Let $r_{11}$ and $r_{22}$ be arbitrary positive, real numbers.  Let
	$(b,c)$ be any pair of non-zero complex numbers satisfying
	\eqref{eqn:b-c-complex-line}. Let $a\in\C$ be chosen arbitrarily.  From
	\eqref{eqn:bc-fix-a-r12}, it follows that $r_{12}$ is a fixed complex
	number.  Since $(a,b,c)$ have been determined, the eigenvalues are
	determined.  Thus, $C \in\C^{2\times 2}$ is also determined.  It
	follows that $A =WRCR^{-1}W^\ast$ satisfies
	\begin{align*}
		A 
		=
		\begin{bmatrix}
			a & c 
			\\ 
			b & a
		\end{bmatrix}
		.
	\end{align*}
	Furthermore, since the diagonals of $R$ are chosen arbitrarily,
	$(b,c)$ must only be chosen to satisfy
	\eqref{eqn:b-c-complex-line}, and $a$ can be chosen freely, it follows
	there are infinitely many $2\times 2$ Toeplitz matrices with the wanted form
	\eqref{eqn:APS-param} using the given $W$.
\end{proof}
\begin{example}[Fully general $2\times 2$ parameterization]
	Since every unitary matrix can be associated to every Toeplitz matrix, we offer 
	an explicit general parameterization.  For an angle $\theta$ 
	build the real-orthogonal Givens rotation 
	\begin{align*}
		W 
		= 
		\begin{bmatrix}
			\cos\theta & -\sin\theta 
			\\
			\sin\theta & \cos\theta
		\end{bmatrix}
		.
	\end{align*}
	We choose $r_{11}=1$ \Wlog  It follows from \eqref{eqn:B-action-on-R} for $i=1$ that for  
	any $(a,b,c)\in\C^3$, an $R$ yielding \eqref{eqn:APS-param} is 
	\begin{align*}
		R 
		=
		\begin{bmatrix}
			1 & a + (b+c)\sin(2\theta)/2
			\\ 
			0 & b\cos^2 \theta - c\sin^2\theta
		\end{bmatrix}
		.
	\end{align*}
	The only condition that cannot be violated in this construction is that 
	\begin{align*}
		r_{22} 
		= 
		b\cos^2 \theta 
		- 
		d\sin^2\theta 
		>
		0
	\end{align*}
	indicating a breakdown necessarily caused by $\bw_1$ being an eigenvector of $A$. 
\end{example}

\subsubsection*{The case $n=3$}
Moving to the $3\times 3$ Toeplitz matrix case introduces the first non-trivial
constraint involving $W$ (see \Cref{remark:Ahat-always-upper-Hessenberg}), as there is one equation of the form
\eqref{eqn:bc-first-ratio-constraints} introducing the rigid invariant involving the
ratio $b/c$. The tridiagonal Toeplitz matrix structure is distinct in this
case; \ie tridiagonal Toeplitz is a proper subclass of Toeplitz for $n>3$.  Thus, we have 
\begin{align*}
	A 
	= 
	\begin{bmatrix}
		a & c &
		\\ 
		b & a & c 
		\\ 
		  & b & a
	\end{bmatrix}
	,
	\quad
	R 
	=
	\begin{bmatrix}
		r_{11} & r_{12} & r_{13}
		\\
		       & r_{22} & r_{23}
		\\     &        & r_{33}
	\end{bmatrix}
	,
	\quad 
	\mbox{and} 
	\    
	\begin{bmatrix}
		A\bb 
		& 
		A^2 \bb	
		& 
		A^3 \bb
	\end{bmatrix}
	= 
	WR 
	. 
\end{align*}
\begin{lemma}
	\label{lemma:3x3-any-unitary-APS} 
	All unitary matrices $W\in\C^{3\times 3}$ are admissible with respect
	to tridiagonal structure.  For any choice of $W$, there exist
	infinitely many tridiagonal Toeplitz matrices admitting the
	factorization \eqref{eqn:APS-param} for this $W$, parameterized by $a$,
	$b$ (or $c$) and $r_{11}$.  For each choice of parameters, there exist
	unique $c$ (or $b$) and $R$ (once $r_{11}$ is chosen freely)
	determining the factorization.
\end{lemma}

\begin{proof}
	For $j=3$, \eqref{eqn:bc-first-ratio-constraints} implies that 
	\begin{align}
		\frac{b}{c} 
		= 
		-
		\frac{
			\bw_3^\ast S \bw_1 
		}{
			\bw_3^\ast S^T \bw_1
		}
		,
		\label{eqn:bc-single-ratio-n3-case}
	\end{align}
	and this is simply a constraint on the ratio $b/c$. We choose $a$
	freely, since from \eqref{eqn:bc-single-ratio-n3-case}, the ratio
	$b/c$ is fixed.  We choose, \Wlog, $b$ freely, which determines $c$. It
	follows that \eqref{eqn:b-c-complex-line} determines
	$r_{22}=r_{11}\FbcHat(2,1)$. Similarly, it follows that
	\eqref{eqn:bc-fix-a-r12} determines $r_{12}=r_{11}\FbcHat(1,1)+a$. The
	mapping $\prn{\FbcHat+aI}R(:,2)=R(:,3)$ yields three equations
	determining 
	\begin{align*}
		r_{13}
		=& 
		\FbcHat(1,1)r_{12} 
		+ 
		\FbcHat(1,2)r_{22}	
		+ 
		a r_{12} 
		\\ 
		r_{23}
		=& 
		\FbcHat(2,1)r_{12} 
		+ 
		\FbcHat(2,2)r_{22}	
		+ 
		a r_{22}
		\\ 
		r_{33}
		=& 
		\FbcHat(3,1)r_{12} 
		+ 
		\FbcHat(3,2)r_{22}	
	\end{align*}
	This fully determines $R$ with no constraint on $W$.
\end{proof}
\begin{example}
	A real $3\times 3$ orthogonal matrix is determined by three parameters, but we use 
	an example involving two parameters for simplicity.  Let $\theta=\pi/3$ and $\phi~=~\pi/4$. 
	We generate an example orthogonal matrix as the product of two associated Givens rotations, 
	\begin{align*}
		W 
		=&
		\begin{bmatrix}
			\frac{1}{2} & -\frac{\sqrt{3}}{2} &
			\\ 
			\frac{\sqrt{3}}{2} & \frac{1}{2} & 
			\\ 
				   && 1
		\end{bmatrix}
		\begin{bmatrix}
			1 && 
			\\ 
			  & \frac{\sqrt{2}}{2} & -\frac{\sqrt{2}}{2} 
			  \\ 
			  & \frac{\sqrt{2}}{2} & \frac{\sqrt{2}}{2}
		\end{bmatrix}
		=
		\begin{bmatrix}
			\frac{1}{2} & -\frac{\sqrt{3}}{2\sqrt{2}} & \frac{\sqrt{3}}{2\sqrt{2}} 
			\\
 			\frac{\sqrt{3}}{2} & \frac{1}{2 \sqrt{2}} & -\frac{1}{2 \sqrt{2}} 
			\\
 			0 & \frac{1}{\sqrt{2}} & \frac{1}{\sqrt{2}}		
		\end{bmatrix}
		.
	\end{align*}
	This allows us to compute 
	\begin{align*}
		W^TS^TW 
		= 
		\begin{bmatrix}
 			\frac{\sqrt{3}}{4} & -\frac{3}{4 \sqrt{2}} & \frac{3}{4 \sqrt{2}} 
			\\
 			\frac{\prn{\sqrt{2}+2 \sqrt{6}}}{8} & \frac{\prn{2-\sqrt{3}}}{8}  & \frac{\prn{\sqrt{3}-2}}{8} 
			\\
 			\frac{2 \sqrt{3}-1}{4 \sqrt{2}} & \frac{\prn{2+\sqrt{3}}}{8}  & \frac{\prn{-2-\sqrt{3}}}{8}
		\end{bmatrix}
		\quad 
		\mbox{and} 
		\quad
		W^TSW 
		= 
		\begin{bmatrix}
			\frac{\sqrt{3}}{4} & \frac{\prn{\sqrt{2}+2 \sqrt{6}}}{8}  & \frac{2 \sqrt{3}-1}{4 \sqrt{2}} 
			\\
 			-\frac{3}{4 \sqrt{2}} & \frac{\prn{2-\sqrt{3}}}{8} & \frac{\prn{2+\sqrt{3}}}{8} 
			\\
 			\frac{3}{4 \sqrt{2}} & \frac{\prn{\sqrt{3}-2}}{8} &\frac{\prn{-2-\sqrt{3}}}{8} 
		\end{bmatrix}
	\end{align*}
	\Wlog we fix $a=2$ and $r_{11}=1$.  The ratio $b/c$ has a rigid invariance due to \eqref{eqn:bc-single-ratio-n3-case},
	\begin{align*}
		\frac{b}{c} 
		=
		-
		\frac{
			\bw_3^T S \bw_1
		}{
			\bw_3^T S^T \bw_1
		} 
		= 
		\frac{3}{1-2\sqrt{3}}
		.
	\end{align*}
	This defines a line passing through the origin in the $b$-$c$ plane on
	which 
    	$(-\bw_3^T S \bw_1,\bw_3^T S^T \bw_1)$ must live.  A choice of $b$ fixes $c$, which in turn
	fixes the remaining entries of $R$.  In
	\Cref{figure:n3-specific-3examples}, we show the line on which $(b,c)$
	must live for this choice of $W$.  For three different points on the
	line (for different choices of $b$), we show the associated tridiagonal
	Toeplitz matrices, all of which admit a common APS factorization
	\eqref{eqn:APS-param} using $W$. 
	\begin{figure}[h]
		\begin{center}
			\includegraphics[scale=0.3]{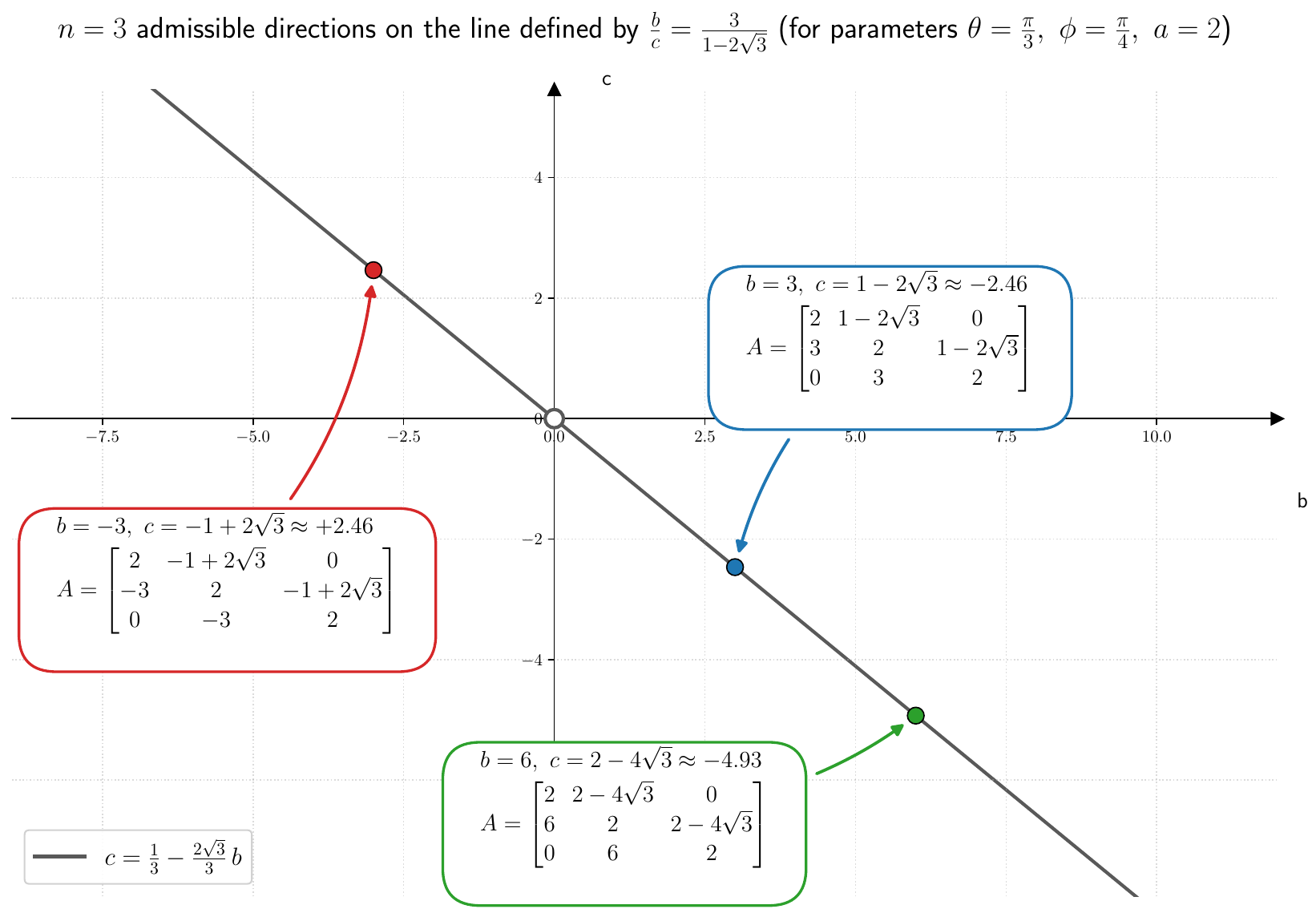}
		\end{center}
		\caption{Three examples of $3\times 3$ tridiagonal matrices for a given value of $b/c$, fixed by a specific choice of $W$.}\label{figure:n3-specific-3examples}
	\end{figure}
	
\end{example}

\subsubsection*{The case $n=4$}
Moving to the $4\times 4$ Toeplitz matrix case introduces additional
non-trivial constraints that are increasingly complicated. The number of
constraints also increases more rapidly with each increase in dimension. Thus,
we have 
\begin{align*}
	A 
	= 
	\begin{bmatrix}
		a & c &   &
		\\ 
		b & a & c &
		\\ 
		  & b & a & c 
		\\ 
		  &   & b & a
	\end{bmatrix}
	,
	\quad
	R 
	=
	\begin{bmatrix}
		r_{11} & r_{12} & r_{13} & r_{14}
		\\
		       & r_{22} & r_{23} & r_{24}
		\\     
		       &        & r_{33} & r_{34} 
		\\
		       &        &	 & r_{44}
	\end{bmatrix}
	,
	\quad 
	\mbox{and} 
	\begin{bmatrix}
		A\bb 
		& 
		A^2 \bb	
		& 
		A^3 \bb
		& 
		A^4 \bb
	\end{bmatrix}
	= 
	WR 
	. 
\end{align*}
\begin{lemma}
	For a unitary matrix $W\in\C^{4\times 4}$ to be admissible \wrt
	tridiagonal Toeplitz structure (with an $R$ constrained appropriately),
	it is necessary for $W$ to satisfy 
	\begin{align}
		-\frac{b}{c} 
		=
		\frac{
			\bw_3^\ast S \bw_1 
		}{
			\bw_3^\ast S^T \bw_1
		}
		= 
		\frac{
			\bw_4^\ast S \bw_1 
		}{
			\bw_4^\ast S^T \bw_1
		}
		=
		\frac{
			\bw_4^\ast S \bw_2
		}{
			\bw_4 S^T \bw_2
		}
		\label{eqn:n4-new-ratio}
		.
	\end{align}
	For any such admissible $W$, there exist infinitely many Toeplitz
	matrices admitting the factorization \eqref{eqn:APS-param} for that
	$W$, parameterized by $a$, $b$ (or $c$) and $r_{11}$. For each choice
	of parameters, there exist unique $c$ (or $b$) and $R$ (with the
	specified diagonal entry) determining the factorization. An admissible
	$W$ must also have a relationship \eqref{eqn:n4-new-ratio} with the
	entries of $R$.
\end{lemma}
\begin{proof}
	As $n$ increases, we only add new constraint equations.  Thus we 
	build on the proof of \Cref{lemma:3x3-any-unitary-APS}.  There are 
	three homogeneous equations 
	\begin{align*}
		\FbcHat(3:4,1)
		=& 
		\bnull 
		\\
		\FbcHat(4,2)
		=& 
		0
		;
	\end{align*}
	these all lead to new homogeneous equations involving $W$ that all can
	be solved for $b/c$, yielding \eqref{eqn:n4-new-ratio}.  Since
	\eqref{eqn:bc-single-ratio-n3-case} is a rigid invariance on $b/c$,
	further homogeneous equations must constrain $W$. Furthermore, the
	mapping 
	\begin{align*}
		\prn{\FbcHat + aI}
		R(:,3)
		= 
		R(:,4)
	\end{align*} 
	yields four equations fixing $r_{14}$, $r_{24}$, $r_{34}$, and $r_{44}$.
	This fixes $R$, and we have proven the rigid constraints on $b/c$ induced 
	by the homogeneous equations necessarily constrain $W$.
\end{proof}
\begin{example}
	A real $4\times 4$ orthogonal matrix is determined by six parameters,
	but we use an example involving three parameters for simplicity and
	because it is a minimal demonstration of $W$ itself becoming
	constrained.  We generate $W$ as the product of three Givens rotations,
	associated to $\theta=\pi/3$, $\phi=-\pi/4$, and $\omega$ left unspecified. 
	\begin{align*}
		W 
		=&
		\begin{bmatrix}
			\frac{1}{2} & -\frac{\sqrt{3}}{2} & 0 & 0 \\
			 \frac{\sqrt{3}}{2} & \frac{1}{2} & 0 & 0 \\
			 0 & 0 & 1 & 0 \\
			 0 & 0 & 0 & 1
		\end{bmatrix}
		\begin{bmatrix}
			 1 & 0 & 0 & 0 \\
			 0 & \frac{1}{\sqrt{2}} & \frac{1}{\sqrt{2}} & 0 \\
			 0 & -\frac{1}{\sqrt{2}} & \frac{1}{\sqrt{2}} & 0 \\
			 0 & 0 & 0 & 1
		\end{bmatrix}
		\begin{bmatrix}
			 1 & 0 & 0 & 0 \\
			 0 & 1 & 0 & 0 \\
			 0 & 0 & c_{\omega} & -s_{\omega} \\
			 0 & 0 & s_{\omega} & c_{\omega} 
		\end{bmatrix}
		\\ 
		=& 
		\begin{bmatrix}
			\frac{1}{2} & -\frac{\sqrt{\frac{3}{2}}}{2} & -\frac{1}{2}\sqrt{\frac{3}{2}} c_{\omega} & \frac{1}{2}\sqrt{\frac{3}{2}} s_{\omega} 
			\\
			\frac{\sqrt{3}}{2} & \frac{1}{2 \sqrt{2}} & \frac{c_{\omega}}{2\sqrt{2}} & -\frac{s_{\omega}}{2 \sqrt{2}} 
			\\
			0 & -\frac{1}{\sqrt{2}} & \frac{c_{\omega}}{\sqrt{2}} & -\frac{s_{\omega}}{\sqrt{2}} 
			\\
			0 & 0 & s_{\omega} & c_{\omega}
		\end{bmatrix}
		,
	\end{align*}
	whereby we have used the shorthand $s_{\omega}\coloneq\sin\omega$ and
	$c_{\omega}\coloneq\cos\omega$. This allows us to compute the sections
	of $W^TS^TW$ and $W^TSW$ involved in the constraint of $W$,
	\begin{align*}
		W(:,3:4)^TS^TW(:,1:2) 
		=&
		\begin{bmatrix}
			\frac{\prn{1+2 \sqrt{3}} c_{\omega}}{4 \sqrt{2}} 
			&
			-\frac{1}{8} \prn{\sqrt{3}-2}c_{\omega}-\frac{s_{\omega}}{\sqrt{2}} 
			\\
			-\frac{\prn{1+2 \sqrt{3}} s_{\omega}}{4 \sqrt{2}} 
			& 
			\frac{1}{8} \prn{\prn{\sqrt{3}-2} s_{\omega}-4\sqrt{2} c_{\omega}} 
		\end{bmatrix}
		\\
		W(:,3:4)^TSW(:,1:2) 
		=&
		\begin{bmatrix}
			-\frac{3 c_{\omega}}{4 \sqrt{2}} 
			&
			 -\frac{1}{8}\prn{2+\sqrt{3}} c_{\omega} 
			\\
			\frac{3 s_{\omega}}{4 \sqrt{2}} 
			&
			 \frac{1}{8}\prn{2+\sqrt{3}} s_{\omega} 
		\end{bmatrix}
		.
	\end{align*}

	For this particular structure of $W$, when studying the first columns
	of these two matrices in the context of \eqref{eqn:n4-new-ratio}, we
	note that $\omega$ is not constrained by the associated homogeneous
	equations since nonzero $s_{\omega}$ and $c_{\omega}$ cancel; and, when
	one of them is zero, the corresponding equation is trivial. These
	equations fix the ratio $b/c=3/\prn{1+2\sqrt{3}}$ via the choices of
	$\theta$ and $\phi$. The real constraint on $\omega$ comes from
	$\FbcHat(4,2)=0$, \ie, 
	\begin{align*}
		0 
		=&
		b 
		\prn{
			\frac{1}{8} \prn{\prn{\sqrt{3}-2} s_{\omega}-4\sqrt{2} c_{\omega}}
		}
		+ 
		c 
		\prn{
			\frac{1}{8}\prn{2+\sqrt{3}} s_{\omega}
		}
		\\ 
		=&
		c\cdot
		\frac{
			 s_\omega\prn{4 \sqrt{3}+1}-6 \sqrt{2} c_\omega
		}{
			4+8\sqrt{3}
		}
		,
	\end{align*}
	which is obtained by substituting the known value of $b/c$ into the
	equation.  We obtain from this $s_\omega/c_\omega = 6 \sqrt{2}/\prn{4
	\sqrt{3}+1}$ which implies 
	\begin{align*}
		\omega
		=
		\arctan\prn{6 \sqrt{2}/\prn{4 \sqrt{3}+1}}
		\approx 
		0.819325
		. 
	\end{align*}
	Thus, the constraints on $W$ fix $\omega$ according to
	\eqref{eqn:n4-new-ratio}. In \Cref{figure:n4-specific-3examples}, we
	show some realizations for different values of $(b,c)$ from the rigid
	constraint line.
    
	\begin{figure}[h]
		\begin{center}
			\includegraphics[width=0.5\textwidth]{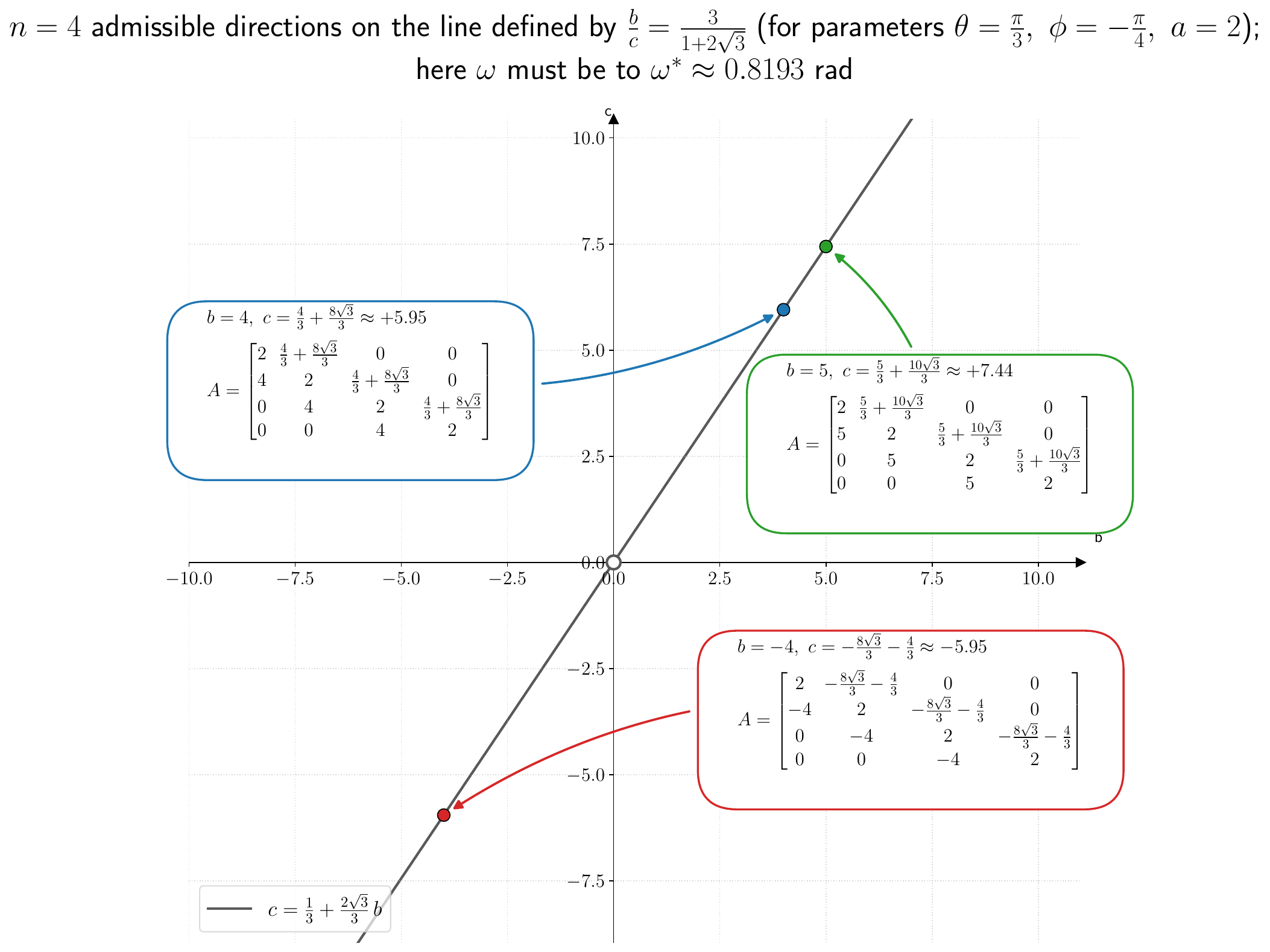}
			\quad
			\includegraphics[width=0.4\textwidth]{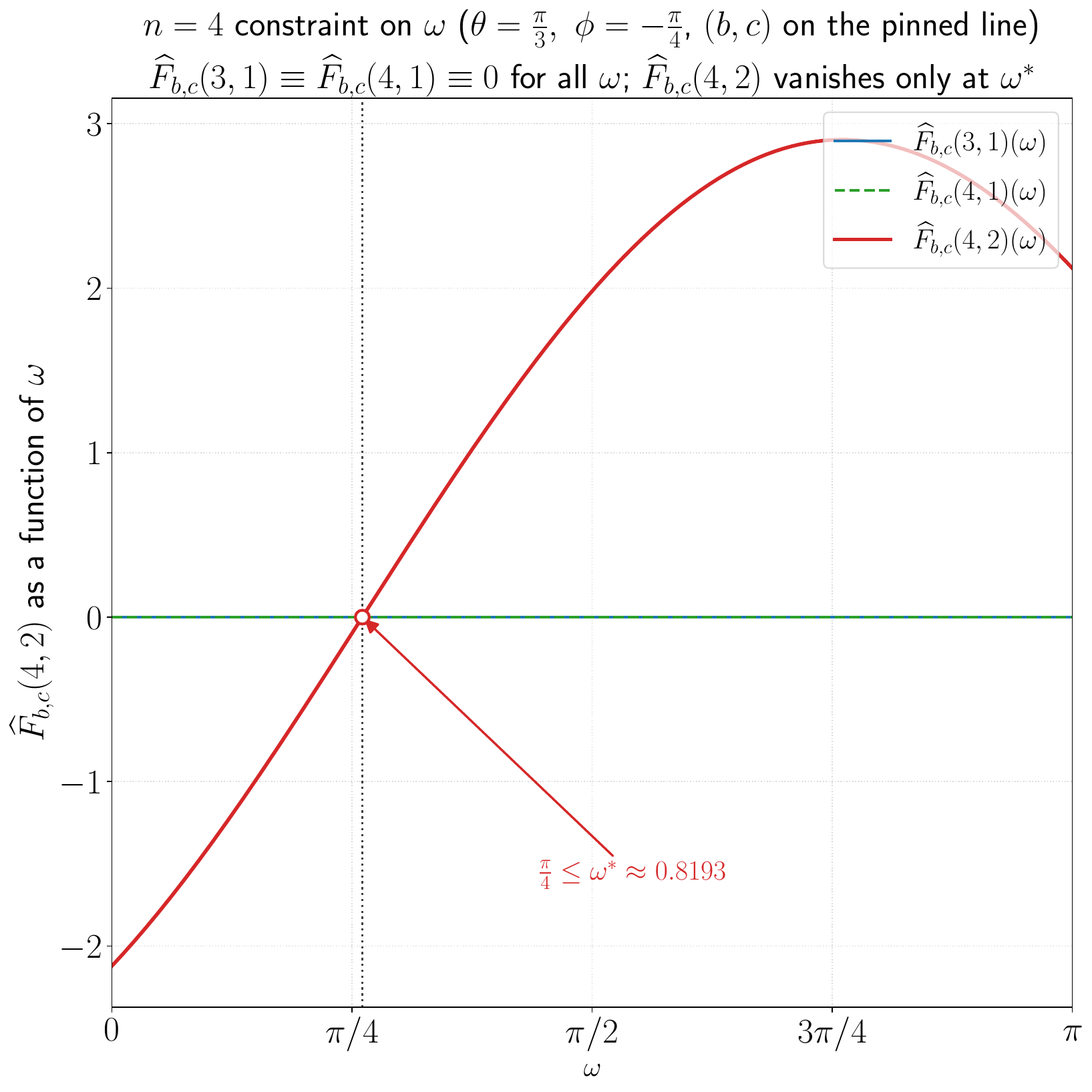}
		\end{center}
		\caption{Three examples of $4\times 4$ tridiagonal matrices for a given value of $b/c$, fixed by a specific choice of $W$.}\label{figure:n4-specific-3examples}
	\end{figure}
\end{example}

\subsubsection*{The case $n>4$}
For $n=4$, we encounter the first case in which $W$ itself is constrained. From
\eqref{eqn:bc-first-ratio-constraints}, we see that the number of
constraints on the columns of $W$ for it to be admissible \wrt tridiagonal
Toeplitz structure increases with dimension.  These constraints all relate to
the invariant $b/c$ determined by $W$, and they specify how columns $3, 4,
\ldots, n$ relate to the first column of $W$.    To discuss the general case,
we use a slightly different approach that scales more cleanly as $n$ gets
larger.

The homogeneous constraints arise due to the upper Hessenberg structure of
$\FbcHat$. To better understand how they constrain the structure of admissible
$W$, we rewrite the equations \eqref{eqn:general-B-W-R-constraints}.  Consider
for fixed $i$ what happens to $\FbcHat(:,i)$ as we increase $n$.  The entries
do not change as we increase $n\rightarrow n+1$; we simply add more homogeneous
equations due to the upper Hessenberg structure of $\FbcHat$.  Since
$i=1,2,\ldots n-1$, the number of homogeneous equations constraining $W$
increases quadratically with $n$.
We
elaborate further on the structure of $W$ for a specified $b/c$ with a lemma, allowing us 
to correctly count the \dof.
\begin{lemma}
	Let $b/c$ be a fixed ratio and $\bw_1\in\C^{n}$ be an arbitrary unit
	vector. Then there exists exactly one unitary matrix $W$ that is
	admissible \wrt a tridiagonal Toeplitz matrix $\tritoep{a}{b}{c}$
	and also has $\bw_1$ as its first column.
\end{lemma}
\begin{proof}
	We we have observed $A\cK_j(A,\bb) = \cK_j(F_{b,c},\bw_1)$ for all
	$j\leq n$.  The matrix $W$ is the $Q$ factor of $AK_n$.  This exists if
	and only if $\bw_1$ is full grade with respect to $F_{b,c}$, and it is
	up to row/column scaling buy unitary scalars.
\end{proof}
Thus, $W$ is completely parameterised by its first column for a fixed $b/c$.  Since 
$\norm[auto]{\bw_1}=1$ and we do not differentiate unitary scalings (\ie, we do not differentiate 
$\e^{i\theta}\bw_1$ from $\bw_1$), $W$ depends on $2n-2$ real parameters, and the remaining 
$n^2-3n+2=(n-1)(n-2)$ parameters are fixed.

\begin{lemma}
	\label{lemma:shifted-range-containment-W}
	An admissible unitary matrix $W$ must satisfy that 
	\begin{align}
		\range 
		\prn{
			W(:,1:i+1) 
		}
		\cap
		\brackets{
			\range 
			\prn{
				S^T W(:,1:i) 
			}
			+
			\range 
			\prn{
				S W(:,1:i) 
			}		
		}
		\neq 
		\braces{\bnull}
		\label{eqn:shifted-range-containment-W}
	\end{align}
\end{lemma}
\begin{proof}
	One observes that from \eqref{eqn:bc-first-ratio-constraints} that
	there exist $\bu,\bv\in\C^{i}$ (\ie, $\bu=b\be_i,\bv=c\be_i$) 
	such that 
	\begin{align*}
		W(:,i+2:n)^\ast 
		\prn{
			S^T W(:,1:i)\bu 
			+ 
			S W(:,1:i)\bv
		}
		=
		\bnull
	\end{align*}
	Since $W$ is unitary, this implies that
	\begin{align*}
		S^T W(:,1:i)\bu + S W(:,1:i)\bv 
		\perp 
		\range\prn{W(:,i+2:n)}
		,
	\end{align*}
	meaning it lies in $\range\prn{W(:,1:i+1)}$, thus proving the result.
\end{proof}

\subsection{Can $W$ from a known $(A,\bb)$ be generated by a second pair?}
We have determined that admissible unitary matrices $W$ are highly constrained
in structure. Consider $W$ generated by a known pair $(A,\bb)$ (thereby
satisfying \Cref{eqn:general-B-W-R-constraints}) for the associated $R$
from \eqref{eqn:APS-param}. Does there exist another pair $(\wtA ,\btilb)$ that
also generates the same $W$?

The answer depends on the dimension $n$.  When $2\leq n \leq3$, we have shown
that infinitely many such $\wtA$ can be constructed for any $W$. However, when
$n\geq 4$, we show that the situation is much more constrained but still
possible.

Since $W$ is considered to have been generated from a known pair $(A,\bb)$, we
treat its entries as constants, and the unknowns are the entries of $R$ (except
for $r_{11}$, which remains free). The ratio $b/c$ is rigidly constrained by $W$;
so we can freely choose \Wlog $b$.  The diagonal shift $a$ remains free. 

Of particular importance is the invariance
\eqref{eqn:bc-first-ratio-constraints}, which tells us that for a given $W$, a
second tridiagonal Toeplitz matrix $\wtA$ generating $W$ must satisfy
$(\tilb,\tilc)=t(b,c)$ for some $t\in\C$ (or $t\in\R$ in the real case). 
Furthermore, for a given $W$, if $r_{11}$, $a$ and \Wlog $b$ are
chosen freely, $c$ is fixed and the generation of the Krylov sequence
\eqref{eqn:B-action-on-R} fixes all entries of $R$, other than $r_{11}$.
\begin{theorem}
	\label{thm:tridiag-toeplitz-APS}
	Let $(A,\bb)$ be a matrix/right-hand side pair with
	$A=\tritoep{a}{b}{c}$ and $bc\neq 0$, such that the APS
	parameterization \eqref{eqn:APS-param} holds with $\bb=\sum_{i=1}^n
	\eta_{i}\bw_i$. 
    Then there exists an affine family of tridiagonal Toeplitz
	matrices 
	\begin{align*}
		\cF_{b/c} 
		\coloneq &
		\braces{t \cdot \tritoep{a}{b}{c} + \alpha I \ |\ \alpha\in\C\ \mbox{and}\ 0\neq t\in\C}
		,
	\end{align*}
	characterized by the ratio $b/c$ and parameterized by $(\alpha,t)$,
	that can generate the residual Arnoldi space associated to the columns
	of $W$. The choices of $\alpha$ and $t$ completely determine a unique
	$R_{t,\alpha}$ associated to that $W$.
\end{theorem}
\begin{proof}
	Most of these results have already been proven. Given a choice of
	$\alpha$ and $t$, we obtain $(\tilb,\tilc)=t(b,c)$. We rescale the
	diagonal and store any discrepancy in $\alpha$. The Krylov iteration
	\eqref{eqn:B-action-on-R} generates the unique $\tilR$ forming
	\eqref{eqn:APS-param}. In the case that $\alpha\neq 0$, it is clear
	from \eqref{eqn:B-action-on-R} that $\tilR$ has no particular structure
	beyond upper triangular with positive diagonal elements to normalize in order to 
	get a unique factorization.  If $\alpha =
	0$, the Krylov subspace generated by \eqref{eqn:B-action-on-R} becomes
	$\cK_j(t\FbcHat, \be_1)=\cK_j(\FbcHat, \be_1)$, which multiplies
	increasing powers of $t$ on each column of $R$ to produce $\tilR$, to
	accommodate the matrix scaling $t$.  In order to maintain $R$ having 
	positive diagonal entries, we scale the rows with unitary scalar values 
	to obtain this property and absorb the scaling into the columns of $W$.
\end{proof}
	Observe that \Cref{thm:tridiag-toeplitz-APS} does not characterize all pairs 
	exhibiting the same convergence, since one cannot simply create a right-hand side $\bb$ with
	the prescribed representation in $W$ without changing the residual subspace; \ie, as we 
	have discussed the relationship is circular.
	\begin{corollary}\label{corollary:limitations-on-rhs}
		Let $(A,\bb)$ and $W$ be defined as in \Cref{thm:tridiag-toeplitz-APS}.
		Then $(\wtA,\btilb)$ with $\wtA\in\cF_{b/c}$ and $\btilb=\beta\wtA^{-1}\bw_1$ 
		(with $\beta\in\C$ being some non-zero scaling factor)
		generates the same \gmres relative residual curve as $(A,\bb)$ if and only 
		if either $\alpha = 0$ or $\mu=\alpha/t$ is such that 
		\begin{align}
			\abs[auto]{\det\prn{T_i + \mu i}}
			= 
			\abs[auto]{\det T_i}
			\quad 
			\mbox{for}
			\quad 
			i 
			= 
			1,2,\ldots, n-1 
			,
			\label{eqn:subdet-equations}
		\end{align}
		where $T_i\in\C^{(n-i)\times(n-i)}$ is the trailing principal submatrix of 
		$\widehat{A}=W^\ast A W$.
	\end{corollary}
	\begin{proof}
		From \Cref{thm:APS-parameterization} it follows, 
		\begin{align*}
			\bh 
			= 
			\widehat{A}^{-1}\be_1 
			\quad 
			\mbox{and}
			\quad 
			\bwth 
			= 
			\prn{t\widehat{A} + \alpha I}^{-1}\be_1 
		\end{align*}
		The case $\alpha=0$ corresponds to a scaling of $A$ for which the result 
		trivially holds. For $\alpha\neq 0$, we consider \Wlog $\wtA=A + \mu I$ due to the 
		scaling invariance of \gmres relative residual convergence.
		From, \eg, \cite{AbderramanMarreroTomeo:2012:1}, we can use determinant
		ratios derived from Cramer's rule to write
		\begin{align*}
			\bh(i) 
			= 
			\sigma 
			\frac{
				\det T_i
			}{
				\det \widehat{A}
			}
			\quad 
			\mbox{and}
			\quad 
			\btilh(i) 
			= 
			\sigma 
			\frac{
				\det \prn{T_i + \mu I}
			}{
				\det\prn{\widehat{A} + \mu I}
			}
			, 
	\end{align*}
	where $\sigma = (-1)^{i+1}\prn{\prod_{k=0}^i \widehat{A}\prn{i-k,i-k-1}}$ does not depend on $\mu$
	and for consistency of the Cramer's-rule-based identities, $\det T_n = 1$.
	Finding $(\wtA,\bwtb)$ producing the same residual curve as $(A,\bb)$ is equivalent to finding
	$\mu\neq 0$ such that for $i=1,2,\ldots, n-1$
	\begin{align}
		\abs[auto]{\btilh(i)} 
		= 
		\gamma 
		\abs[auto]{\bh(i)}
		\label{eqn:htil-scaling}
	\end{align}	
	for some scaling $\gamma$.  The equality for the last coordinate $\btilh(n)=\gamma\bh(n)$ 
	yields 
	\begin{math}
		\gamma 
		= 
		\frac{
			\abs[auto]{\det\widehat{A}}
		}{
			\det\prn{\widehat{A}+\mu I}
		}
		.
	\end{math}
	Plugging this into \eqref{eqn:htil-scaling} and cancelling terms yields the result. 
\end{proof}
From this result, we conclude that it is exceedingly rare that two different pairs $(A,\bb)$ and $(\wtA,\btilb)$ generating the same $W$ and the same \gmres relative residual curve exist. The equation 
associated to $i=n-1$ yields that $\mu$ must satisfy 
\begin{align*}
	\abs[auto]{\widehat{A}(n,n) + \mu} 
	= 
	\abs[auto]{\widehat{A}(n,n)}
	.
\end{align*}
Should such a $\mu$ exist, it would need to satisfy \eqref{eqn:subdet-equations} for 
$i=1,2,\ldots n-2$.

\subsection{The scaled Jordan block case}
\label{subsection:dealing-with-jordan-blocks}
Although some computations simplify and
the spectral structure collapses, the theory developed in this manuscript mostly
extends without issue to the case of $bc=0$, which can be explored $\Wlog$ for
$b=0$, \ie for scaled Jordan blocks.  In this case,  
\begin{align*}
	A
	=&
	\tritoep{a}{0}{c}
	=
	c\tritoep{a/c}{0}{1}
	= 
	aI + cS
\end{align*}
is a 
scaled Jordan block.  The eigenvalues collapse onto a single degenerate eigenvalue $a$ with geometric 
multiplicity of one.  

Structurally, we observe that for a Jordan block, the situation becomes simpler
due to \emph{scaling invariance} of Krylov subspaces (in addition to the
aforementioned shift invariance); namely $\cK_j(aI+cS, \bw_1)=\cK_j(S,\bw_1)$.
Thus, all Jordan blocks generate the same $W$, if given the same starting
vector.  Thus, they form a single family sharing the same set of admissible
unitary matrices \wrt tridiagonal Toeplitz structure.  The parameter $a$ is
still free, and it now determines the single eigenvalue.  Framing this in the
same language as for the non-Jordan case, this family is defined by the
constraint $(b,c)$ lies on the line $(0,t)$.  Structurally, the APS
parameterization \eqref{eqn:APS-param} does not change much, since $C$ is
defined via the still-degree-$n$ characteristic polynomial.  Constraints
determining the family of admissible $W$ reduce from
\eqref{eqn:general-B-W-R-constraints} to
\begin{align}
	W(:,i+2:n)^\ast S W(:,1:i)
	R(1:i,i)
	=
	\bnull
	\quad 
	\mbox{for}
	\quad 
	i=1,2,\ldots, n-1
	,
	\label{eqn:general-B-W-R-constraints-jordan-null}
\end{align} 
whereby $c\neq 0$ has dropped out, demonstrating that the constraint on admissible $W$ does not depend at all on 
which Jordan block.  It merely depends on the Jordan structure itself.

\subsection{Attainable convergence}
\label{subsection:admissible-attainable}
We have described how the limited real \dof determining the pair $(A,\bb)$ in
the case that $A=\tritoep{a}{b}{c}$ constrains the APS-parameterization.  We
build on this to complete the theory by describing admissible and attainable
convergence behavior for the family of tridiagonal Toeplitz matrices, which concerns 
regularity of $\bF$.

\begin{remark}\label{remark:real-analysis-of-F}
	For the analysis of regularity of $\bF$, we reduce down to the case of
	real $A$ and $\bb$.  This simplifies the calculations and keeps the
	exposition clear. In this setting, we observe that there are two
	distinct parameter regimes: $bc > 0$ and $bc < 0$.  We reparameterize 
	the matrix to carry out analysis similar to \cite{LiesenStrakos:2004:1},
	but we encode the sign itself as a separate parameter which is treated as 
	fixed and thus passes through the analysis.  Thus, our analysis holds for 
	both cases, and the case of the Jordan block ($bc=0$) is treated separately,
	as the boundary case between the two regimes.
\end{remark}

For our analysis, it is convenient to re-parameterize in terms of new variables
$(a, \tau, \xi)$, where for $A = \tritoep{a}{b}{c}$, we have 
$\varepsilon=\sign\prn{bc}$,
$\tau = \sqrt{\abs[auto]{bc}}$ and
$\xi=\sqrt{\abs[auto]{b/c}}$ so that we can rewrite 
$A = \tritoep{a}{\tau\xi}{\varepsilon\tau/\xi}$.  This
type of reparameterization is in line with analysis in works such as
\cite{LiesenStrakos:2004:1} that observe that the convergence of \gmres for
these systems is more naturally analyzed in terms of $a$ (the center of the
eigenvalue cluster), $bc$ (the spread of the eigenvalues) and $b/c$ (the
non-normality).  Parameterizing the latter two using an encoded sign $\varepsilon$, 
is for ease of later calculation.  We avoid unnecessary consideration of
matters concerning taking square roots of negative numbers.  
Thus, the convergence can be expressed as 
a function of $(a, \tau, \xi)$ and of $\bw_1$; the analysis holds for $\varepsilon=\pm 1$ 
while treating the sign as a constant passing through the calculations.

The expressions in \eqref{eqn:gmres-gram-determinant-subspace-angle} give us a
direct expression for the residual norms as a function of these parameters.
\Wlog we consider the relative residual behavior and assume
$\norm[auto]{\br_0}=1$. We prefer to analyze the Gram determinant formulation,
but we do not foreclose the possibility of using the subspace angle formulation
to carry out a similar analysis.  Since \gmres relative residual convergence,
is invariant \wrt $\norm[auto]{\bb}$, we consider only $\bb\in\cS^{n-1}$ (\ie, from
the unit sphere); thus $\bF: \R^3\times \cS^{n-1}\rightarrow \R^n$ be a
vector-valued function whose components $F_1, F_2, \ldots, F_{n-1}$ are defined
via 
\begin{align*}
	F_k
	=
	F_k(a, \tau, \xi, \bb) 
	\coloneq
	\frac{
		\det\prn{K_{k+1}^\ast K_{k+1}}
	}{
		\det\brackets{\prn{AK_k}^\ast AK_k}
	}
	\quad 
	\mbox{where} 
	\quad 
	A 
	= 
	\tritoep{a}{\tau\xi}{\varepsilon\tau/\xi}
	,
\end{align*}
with $K_j$ denoting the first $j$ Krylov basis vectors \eqref{eqn:krylov-matrix}. It is clear that $\bF$
maps parameters to residual 
curves. Investigating the regularity of $\bF$ gives us insights into what residual curves are attainable. 

Following \Cref{remark:real-analysis-of-F}, we analyze the vector-valued
function $\bF:\R^{n+2}\rightarrow \R^{n-1}$. Thus, the Jacobian evaluated at a
point $\bF'(a, \tau, \xi, \bb)\in\R^{(n-1)\times (n+2)}$ is a matrix.  We split
the analysis into two parts: the columns of $\bF'$ associated to the $A$ (\ie,
\wrt $a$, $\tau$, and $\xi$) and the part associated to (the free parameters
determining) $\bw_1$.

\subsubsection*{Differentiating the residual curve \wrt $a$, $\tau$, and $\xi$}
For differentiating each $F_k$, we use the logarithm trick to separate the
quotient into a difference, then apply Jacobi's formula 
\begin{align}
	\log F_k 
	=&
	\log \det\prn{K_{k+1}^T K_{k+1}}
	- 
	\log \det\brackets{\prn{AK_k}^T AK_k}
	\nonumber
	\\ 
	\iff 
	\frac{\partial}{\partial\theta} 
	F_k 
	=& 
	F_k 
	\frac{\partial}{\partial\theta} 
	\brackets{
		\trace \prn{K_{k+1}^T K_{k+1}} - \trace \brackets{\prn{AK_k}^T AK_k}
	}
	\label{eqn:log-based-residual-norm-deriv}
\end{align}
where $\theta$ is any variable, and we have used the well-known differentiation
formula for $\log$ of a function. Thus, every entry of the Jacobian is obtained
via a computation of this sort for $\theta = a,\ \tau,\ \xi$.  For notation
simplification during differentiation, we denote 
\begin{align}
	P_k 
	\coloneq 
	K_{k+1}^T K_{k+1} 
	=& 
	\prn{
		\bb^T(A^{j-1})^TA^{\ell-1}\bb
	}_{j,\ell}
	\nonumber
	\\ 
	Q_k 
	\coloneq 
	(AK_k)^T AK_k 
	=& 
	\prn{
		\bb^T(A^j)^TA^\ell\bb
	}_{j,\ell}
	.
	\label{eqn:Pk-and-Qk-definition}
\end{align}
Since differentiation is linear, it passes through the trace and is applied directly to the entries of $P_k$ and $Q_k$. 
\begin{lemma}\label{lemma:Al-gen-deriv}
	For any parameter $\theta$, we have the matrix powers derivative 
	\begin{align}
		\frac{\partial}{\partial\theta} 
		A^j 
		= 
		\sum_{i=0}^{j-1} A^i
		\cdot 
		\frac{\partial}{\partial\theta}A
		\cdot 
		A^{j-1-i}
		.
		\label{eqn:Al-gen-deriv}
	\end{align}
\end{lemma}
\begin{proof}
	This is a standard result proved by induction.  The base case is given by the usual product rule, 
	\begin{math}
		\frac{\partial}{\partial\theta} 
		A^2 
		= 
		\frac{\partial}{\partial\theta}
		A 
		\cdot 
		A 
		+ 
		A 
		\cdot 
		\frac{\partial}{\partial\theta} 
		A
		,
	\end{math}
	and for the inductive step, we show it for $A^j$ assuming it holds for $A^{j-1}$,
	\begin{align*}
		\frac{\partial}{\partial\theta} 
		A^j 
		=&
		\frac{\partial}{\partial\theta}
		A^{j-1} 
		\cdot 
		A 
		+ 
		A ^{j-1}
		\cdot 
		\frac{\partial}{\partial\theta} 
		A
		\\ 
		=&
		\prn{
			\sum_{i=0}^{j-2} A^i
			\cdot 
			\frac{\partial}{\partial\theta}A
			\cdot 
			A^{j-2-i}
		}
		\cdot 
		A 
		+ 
		A ^{j-1}
		\cdot 
		\frac{\partial}{\partial\theta} 
		A
		,
	\end{align*}
	which when multiplied and combined into one sum completes the proof.
\end{proof}
A key fact that is helpful in computing these derivatives is understanding for
which $\theta$ $A$ commutes with $\frac{\partial}{\partial\theta} A$.
\begin{corollary}\label{corollary:Al-spec-deriv}
	For $\theta= a,\ \tau,\ \xi$, \eqref{eqn:Al-gen-deriv} simplifies to 
	\begin{align}
		\Dparam{a}A^\ell 
		=& 
		\ell A^{\ell-1}
		\label{eqn:Al-deriv-a}
		\\ 
		\frac{\partial}{\partial\tau}A^\ell 
		=& 
		\frac{\ell}{\tau}
		\prn{
			A^\ell 
			- 
			a A^{\ell-1}
		}
		\label{eqn:Al-deriv-tau}
		\\ 
		\frac{\partial}{\partial\xi}A^\ell 
		=& 
		\sum_{i=0}^{\ell-1} A^i
		\cdot
		\prn{
			\tau S^T - \varepsilon\frac{\tau}{\xi^2}S
		}
		\cdot 
		A^{\ell-1-i}
		\label{eqn:Al-deriv-xi}
		.
	\end{align}
\end{corollary}
\begin{proof}
	We remind the reader that we represent
	$A=\tritoep{a}{\tau\xi}{\varepsilon\tau/\xi}$. For the derivative \wrt $a$, one
	simply observes that $\Dparam{a}A=I$; and since it
	commutes with $A$, this collapses \eqref{eqn:Al-gen-deriv} yielding the
	result. 

	Similarly, $\Dparam{\tau}A=\xi S^T + \frac{1}{\xi} S$. Observing that 
	$b=\tau\xi\iff\xi=b/\tau$, and $c=\varepsilon\tau/\xi\iff1/\xi = \varepsilon c/\tau$, we can rewrite 
	$\Dparam{\tau}A= \frac{1}{\tau}\prn{A-aI}$ which when inserted into 
	\eqref{eqn:Al-gen-deriv}, collapses it yielding the result. 

	The derivative \wrt $\xi$ does not yield any simplification. This follows from the fact that 
	$\Dparam{\xi}A=\tau S^T - \frac{\tau}{\xi^2}S$, which we simply insert into 
	\eqref{eqn:Al-gen-deriv} since it does not commute with $A$.  This completes the proof.
\end{proof}
The equations (\ref{eqn:Al-deriv-a},\ref{eqn:Al-deriv-tau}) demonstrate that $A$ commutes with its derivatives 
\wrt to $a$ and $\tau$, but \eqref{eqn:Al-deriv-xi} shows it does not commute with its derivative \wrt to $\xi$, and this has 
consequences for how we compute elements of the $\bF'$.  Thus, we treat differentiation \wrt $\xi$ separately, as it requires 
additional development. For differentiation \wrt $a$ and $\tau$, another lemma follows directly from \Cref{corollary:Al-spec-deriv}. 
\begin{lemma}\label{lemma.resid-norm-deriv-a-tau}
	The derivatives of $\norm[auto]{\br_k}$ \wrt to $a$ and $\tau$ admit the expressions 
	\begin{align}
		\Dparam{a}
		\norm[auto]{\br_k}^2
		=&
		-2
		\norm[auto]{\br_k}^2
		\by_k(1)		
		\\ 
		\Dparam{\tau}
		\norm[auto]{\br_k}^2
		=&
		2
		\norm[auto]{\br_k}^2
		\frac{a}{\tau} 
		\by_k(1) 
		,
	\end{align}
	where $\by_k$ is the solution of the small \gmres least-squares problem \wrt the Krylov power basis $K_k$
	\begin{align}
		\by_k 
		= 
		\argmin_{\by\in\R^k}\norm[auto]{A K_k\by - \bb}
		.
		\label{eqn:gmres-small-ls}
	\end{align}
\end{lemma}
For exposition clarity, we put the proof in \Cref{section:resid-norm-deriv-a-tau-PROOF}.

\begin{corollary}
	The gradients $\frac{\partial\bF}{\partial a}$ and $\frac{\partial\bF}{\partial \tau}$ are 
	collinear.
\end{corollary}

Since the derivative \wrt $\xi$ does not commute with $A$, we adjust our approach for the normality parameter. 
Important for encoding this lack of commuting is the commutator, whose definition we recall to the reader.  
\begin{definition}
	Let $X$ and $Y$ be two square matrices of dimension greater than one.  We denote as 
	\emph{the commutator} the residual of commutativity 
	\begin{align*}
		[X,Y] 
		=
		XY-YX 
		,
	\end{align*}
	which is zero in the case that $X$ and $Y$ commute multiplicatively.
\end{definition}
The commutator is useful in a variety of fields; 
see, \eg \cite[Chapter 1.6]{SakuraiTuan:1994:1}.  Among its many interesting properties, it has a product rule 
reminiscent of the differentiation product rule; \ie 
\begin{align}
	[M,XY] 
	= 
	[M,X]Y 
	+ 
	X[M,Y]
	\label{eqn:commutator-product-rule}
	,
\end{align}
which we have shown in \eqref{eqn:commutator-product-rule} in the right-hand argument, \Wlog
From this, an identity involving integer powers can be derived
\begin{align}
	[M,A^\ell] 
	= 
	\sum_{i=0}^{\ell-1} 
	A^i 
	[M,A]
	A^{\ell-i-1}
	\label{eqn:commutator-product-rule-integer-powers}
	.
\end{align}
Let $A = Z\Lambda Z^{-1}$ be the eigendecomposition, with $\Lambda$ and $Z$ having entries as described 
in \eqref{eqn:Toeptliz-eig}. Applying the product rule,  we differentiate the eigendecomposition 
\begin{align*}
	\Dparam{\xi} 
	A 
	= 
	\prn{\Dparam{\xi}Z}\Lambda Z^{-1}
	+
	Z\prn{\Dparam{\xi}\Lambda} Z^{-1}
	+
	Z\Lambda\prn{\Dparam{\xi}Z^{-1}}
	= 
	\prn{\Dparam{\xi}Z}\Lambda Z^{-1} 
	+
	Z\Lambda\prn{\Dparam{\xi}Z^{-1}}
	,
\end{align*}
since $\Lambda$ is purely a function of $a$ and $\tau$ in this
parameterization. From the formula for entries of $Z$ from
\eqref{eqn:Toeptliz-eig}, one sees that it is simply constant multiples of
powers of $\xi$.  Let us denote the $n$-diagonal indexing matrix by 
\begin{align}
	N 
	\coloneq 
	N_n 
	\label{eqn:n-diag-indexing-matrix}
	.
\end{align}
At this matrix level, this means that 
\begin{align*}
	\Dparam{\xi} Z 
	= 
	\frac{1}{\xi} N Z 
	.
\end{align*}
Differentiating the equation $ZZ^{-1}=I$ and solving yields 
\begin{align*}
	\Dparam{\xi} Z^{-1} 
	=
	-Z^{-1}
	\prn{\frac{1}{\xi} N Z} 
	Z^{-1} 
	= 
	-\frac{1}{\xi} 
	Z^{-1} N 
	,
\end{align*}
which when substituted back into the expression for $\Dparam{\xi}A$ yields 
\begin{align}
	\Dparam{\xi} 
	A 
	= 
	\frac{1}{\xi}
	\prn{
		NA - AN
	} 
	= 
	\frac{1}{\xi} 
	[N,A] 
	. 
\end{align}
\begin{lemma}
	It follows that 
	\begin{align*}
		\Dparam{\xi} A^\ell 
		= 
		\frac{1}{\xi} 
		[N, A^\ell]
		,
	\end{align*}
	for integers $\ell \geq 0$.
\end{lemma}
\begin{proof}
	From \Cref{lemma:Al-gen-deriv}, it follows that $\Dparam{\xi}A^\ell = \sum_{i=0}^{\ell-1} A^i\prn{\frac{1}{\xi}[N,A]}A^{\ell-i-1}$,
	and the proof is completed observing that this is exactly the identity concerning commutators of integer powers \eqref{eqn:commutator-product-rule-integer-powers}.
\end{proof}
This gives us the tools we need to differentiate \wrt $\xi$. 
\begin{lemma}\label{lemma:resid-norm-deriv-xi}
	The derivative of the residual norm \wrt the normality parameter $\xi$ can be expressed as 
	\begin{align}
		\Dparam{\xi} 
		\norm[auto]{\br_k}^2 
		= 
		\frac{2}{\xi} 
		\cdot 
		\frac{
			\br_k^T [N, \Psi_k(A)]b
		}{
			\norm[auto]{\br_k}
		}
		,
	\end{align}
	where $\Psi_k\prn{z}$ is the $k$-th \gmres residual polynomial.
\end{lemma}
For exposition clarity, we put the proof in \Cref{section:resid-norm-deriv-xi-PROOF}.

\subsubsection*{Differentiating the residual curve \wrt $\bb$}
In comparison to differentiating \wrt the parameters determining $A$,
differentiating \wrt $\bb$ is relatively straightforward. We note that
equivalently, we could differentiate \wrt $\bw$, but differentiation \wrt to
$\bb$ works better in the context of our analysis. In either case these
parameters enter only through a bi-linear dependence.  Consider that 
\begin{align*}
	\Dparam{\bb}\norm[auto]{\br_k}^2 
	=& 
	\Dparam{\bb}\prn{\Psi_k\prn{A}\bb}^T\prn{\Psi_k\prn{A}\bb}
	\\
	\iff
	2\norm[auto]{\br_k}\Dparam{\bb}\norm[auto]{\br_k}
	=&
	\prn{\Dparam{\bb}\Psi_k\prn{A}\bb}^T\br_k 
	+ 
	\br_k^T\prn{\Dparam{\bb}\Psi_k\prn{A}\bb}
	\\
	\iff
	\Dparam{\bb}\norm[auto]{\br_k}
	=&
	\frac{1}{2}\cdot 2\br_k^T\Psi_k\prn{A} 
	= 
	\br_k^T\Psi_k\prn{A}
\end{align*}

\subsection{Regularity and the convergence envelope}\label{subsecton:regularity-conv-env}
Having the Jacobian $\bF'$ completely in hand opens up the possibility for many further explorations 
that are beyond the scope of our work; see \Cref{section:conclusions} for ruminations on some paths forward.  

We instead focus on closing the theory developed in \Cref{section:gmres-toeplitz-parameterization}, for a given, fixed 
$A=\tritoep{a}{\tau\xi}{\tau/\xi}$, \emph{can any admissible convergence curve in convergence envelope for $A$ be attained}, where we define the convergence envolpe as the continuum of  
admissible convergence curves existing between the best- and worst-case convergence curves,
as defined in, \eg, \cite{Toh:1997:1}.
\begin{remark}\label{remark:APS-non-fixed-A}
	Note that this is a slightly different question than the one treated in
	\cite{GreenbaumPtakStrakos:1996:1} and its extension to the
	APS-parameterization in \cite{ArioliPtakStrakos:1998:1}. That theory
	and by extension the theory we developed describes the simultaneous
	construction of pairs $(A,\bb)$ yielding a specific convergence curve.
	In the case of $A=\tritoep{a}{\tau\xi}{\tau/\xi}$, this corresponds to  
	an appropriate level of non-normality being dialed to enable the construction 
	of some right-hand side inducing that convergence pattern. The present follow-up
	question is asking what is possible if $A$ is fixed.
\end{remark}
Consider that it is well established that $a$ and $\tau$ determine the center
and cluster tightness of the eigenvalues, while $\xi$ controls the
non-normality.  Together, they determine the field of values $W(A)$.  Elman in
his thesis \cite{Elman:1982:1} showed the now well-known result that \gmres
will not exhibit any stagnating iterations for any right-hand side if $0\notin
W(A)$.  Thus, the upper boundary may be total stagnation if $0\in W(A)$, but it will be 
strictly decreasing with an upper bound described by Elman \cite{Elman:1982:1} otherwise. 

Let $\S^{n-1}$ denote the unit sphere in $\R^n$. For the best case lower-bound
convergence, we must consider the pathological case of early convergence.  We
can always construct $\bw_1\in\S^{n-1}$ with representation in the eigenbasis
having only $m\ll n$ non-zero components, meaning \gmres will converge to the
exact solution in at most $m$ iterations. Such right-hand sides live by
construction in at most dimension $n-1$ subspaces spanned by subsets of
eigenvectors; \ie they are a set of measure zero.  We exclude these right-hand
sides and consider only right-hand sides $\bw_1$ for which
$\cK_n(A,\bw_1)=A\cK(A,\bb)$ has \emph{full grade}; \ie it has dimension $n$.
The hyperplanes defined by proper subsets of eigenvectors divide the unit
sphere into components containing candidate full-grade $\bw_1$. We narrow the
question; within the convergence envelope of worst- and best-case full-run
\gmres convergence, are all admissible convergence curves attainable?

\begin{lemma}\label{lemma:F-Jacobian-singular}
	Let $\bF'_{\bb}(\bb)\in\R^{(n-1)\times(n-1)}$ denote the square section
	of the Jacobian $\bF'$ associated to convergence specification via
	choice of $\bb$.  For a given $\bw_1$, $F'_{\bb}(\bb)$ is singular
	if the resulting \gmres iteration exhibits one of the
	following: 
	\begin{itemize}
		\item early convergence to the exact solution;
		\item best or worst case convergence being attained at every iteration;
		\item one or more stagnating iterations.
	\end{itemize}
\end{lemma}
\begin{proof}
	These are proven directly.  

	\emph{For early termination}: We consider \Wlog the case that \gmres
	terminates one iteration early, at iteration $n-1$.  Since
	$\br_{n-1}=\bnull$, \eqref{eqn:trace-component-2} implies that last
	column of $F'_{\bb}(\bb)$ is a zero column, meaning the Jacobian is
	singular. 

	\emph{For best/worst-case convergence}: \Wlog we discuss in terms of worst-case 
	convergence.  If the theoretical worst-case convergence for a given $A$ were attainable 
	for some $\bb^\star$, it would be a critical point for each $k$ of $\Dparam{\bb}\norm[auto]{\br_k}$. 
	This produces a zero row in $\bF'_\bb$, meaning it is singular.  Note: the best/worst case 
	convergence patterns may not be attainable for a given $A$.  This simply observes that if 
	it is attainable, the Jacobian $\bF'_\bb$ at that point is singular.

	\emph{For stagnation}: We observe that stagnation can occur in the case
	that $0\in W(A)$. \Wlog we consider a single stagnating iteration.  In
	that case, one observes from \eqref{eqn:trace-component-2} that two
	adjacent columns of $F'_{\bb}(\bb)$ are equal.  Thus the Jacobian is
	singular. Observe that total stagnation (\ie at every iteration until
	the last) can actually be a special, attainable case of worst-case
	convergence.
%
\end{proof}
\begin{remark}\label{remark:possible-saddle-points-etc}
	We note that these are sufficient conditions for a singular $\bF'_\bb$,
	but they are not necessary. It is possible that there could be saddle
	points, cusps, or other such phenomena.  The full nature of this
	landscape remains an open question, discussed further in
	\Cref{section:conclusions}.
\end{remark}
The theory we have built up does not allow us to fully answer the question, but we can answer it locally. 
\begin{theorem}
	Let $A=\tritoep{a}{\tau\xi}{\tau/\xi}$ be fixed, and 
	let $\bb$ be full grade such that $\bF'_\bb(\bb)$ is non-singular.  Then there exist neighborhoods 
	$U$ of $\bb$ in $\S_{n-1}$ and $V$ of the convergence curve $\bF(a, \tau, \xi, \bb)$ wherein 
	$\bF(a, \tau, \xi, \cdot):U\rightarrow V$ is invertible and bijective (indeed, a diffeomorphism); \ie every convergence curve in $V$ is attainable.
\end{theorem}
\begin{proof}
	This is a direct application of the Inverse Function Theorem. 
\end{proof}

\section{Demonstrations of the theory}
\label{section:numerical-demonstrations}

It is difficult to construct large-scale examples demonstrating these results; so we use examples for dimension $n=4$, since we 
established in \Cref{section:gmres-toeplitz-parameterization} that this is the smallest dimension that is fully general 
in terms of the APS-parameterization structure.

\subsection{
	Convergence profile for 
	\begin{math}
		\bb 
		=
		\frac{1}{2}
		\begin{bmatrix}
			1 & 1 & 1 & 1
		\end{bmatrix}^T
	\end{math}
}

\begin{figure}[h]
	\begin{center}
		\begin{tikzpicture}
			\begin{axis}[width=8cm, height=6cm, ymode=log, xlabel={iteration $k$},
			    ylabel={$\|\br_k\|/\|\br_0\|$}, xtick={0,1,2,3}, xmin=0, xmax=3]
			  \addplot[black, thick, mark=none] coordinates {(0,1.0000) (1,0.2562) (2,0.2494) (3,0.2483)};
			  \addplot[black, only marks, mark=*, mark size=2.2pt] coordinates {(0,1.0000)};
			  \addplot[red, only marks, mark=triangle*, mark size=3pt] coordinates {(1,0.2562) (2,0.2494) (3,0.2483)};
			\end{axis}
		\end{tikzpicture}
	\end{center}
	\caption{
		Relative residual norms for \gmres applied for
		$A=\tritoep{0.8}{1.6}{0.625}\in\R^{4\times 4}$ (\ie $a=0.8$,
		$\tau = 1$, $\xi=1.6$) for right-hand side 
		\begin{math}
			\bb 
			=
			\frac{1}{2}
			\begin{bmatrix}
				1 & 1 & 1 & 1
			\end{bmatrix}^T
			.
		\end{math}
		We seek other right-hand sides $\btilb\in\R^4$ producing exactly this same convergence curve.
		\label{figure:gmres-curve-n4-bones}
	}
\end{figure}

\begin{figure}
	\begin{center}
		\includegraphics[width=0.65\textwidth]{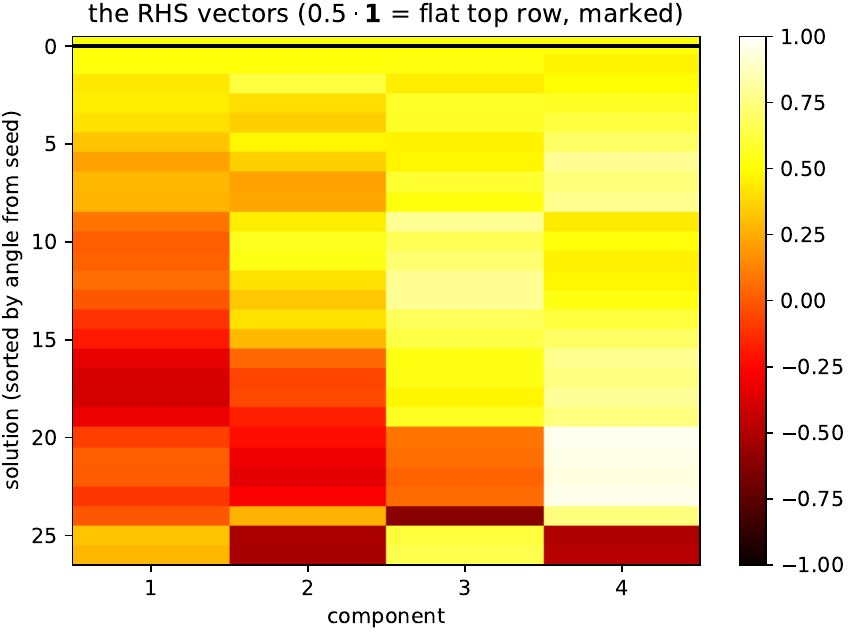}
	\end{center}
	\caption{
		A colormap showing the value of different right-hand sides for
		which \gmres exhibits same convergence as for
		$\bb=\frac{1}{2}\bone$.
		\label{figure:rhss-same-gmres-curve}
	}
\end{figure}

We demonstrate how the Jacobian $\bF'_\bb$ can be used in a Newton iteration
to find right-hand sides for which \gmres produces a prescribed convergence
curve. Let $A=\tritoep{0.8}{1.6}{0.625}~\in~\R^{4\times 4}$; \ie 
$a=0.8$, $\tau = 1$, and $\xi=1.6$.  Other right-hand sides are obtained by 
choosing random starting vectors and running a damped Newton iteration using $\bF'_\bw$,
normalizing to remain on the unit sphere after each step, an example of Newton's Method 
on the unit sphere with retraction back to the sphere after every step \cite[Chapter 6]{AbsilMahonySepulchre:2008:1}.
\footnote{
	We use this method out of the box, and it was effective at finding new
	right-hand sides with the prescribed convergence.  It was out of the
	scope of this work to investigate its convergence properties in the 
	context of \gmres convergence specification.
}
We performed this Newton iteration for 400 random starting points. We do not 
claim that this exhausts all possible right-hand sides generating the same curve,  
as our goal is to just illustrate the theory.  We report that when tested, these 
right-hand sides produced \gmres relative residual curves with norms agreeing to four or 
more decimal places.

To illustrate, we match the convergence of the seed system generating the
wanted \gmres curve,
\begin{align*}
	\bb 
	=& 
	\begin{bmatrix} 
		0.5 
		& 
		0.5 
		& 
		0.5
		& 
		0.5 
	\end{bmatrix}^T
	\quad
	\bw_1 
	= 
	\begin{bmatrix}
		 0.2790 
		 & 
		 0.5922 
		 & 
		 0.5922 
		 & 
		 0.4698 
	\end{bmatrix}^T
	\\
	W 
	= &
	\begin{bmatrix}
		 0.2790 & -0.3551 & -0.2213 & 0.8644 
		 \\ 
		 0.5922 & -0.6648 & 0.0424 & -0.4534 
		 \\ 
		 0.5922 & 0.5630 & -0.5668 & -0.1050 
		 \\ 
		 0.4698 & 0.3391 & 0.7924 & 0.1906 
	\end{bmatrix}
	R 
	= 
	\begin{bmatrix}
		 2.5540 & 6.5562 & 16.5426 & 40.8939 
		 \\ 
		 0 & 0.8834 & 3.0781 & 9.0778 
		 \\ 
		 0 & 0 & 1.1303 & 3.6288 
		 \\ 
		 0 & 0 & 0 & 0.1267 
	\end{bmatrix}
	,
\end{align*}
which produces residual norms (rounded to four decimal places) $\braces{1,
0.2562, 0.2494, 0.2483, 0}$. Running damped Newton iterations (with convergence
tolerance $10^{-10}$ and a damping factor $s=2^{-k+1}$ at iteration $k$) with
random starting vectors produced 26 distinct right-hand sides exhibiting the
same \gmres convergence curve. In \Cref{figure:cosines-prescribed-rhss}, we
display for each generated right-hand side the cosine of its angle with $\bb$
as well as a nearest-neighbor angle to quantify isolation.

\begin{figure}
	\begin{center}

		\begin{tikzpicture}
			\begin{axis}[width=9cm, height=4.5cm, xlabel={generated RHS index}, ylabel={$\abs[auto]{\cos\prn{\bb,\bb^{(i)}}}$},
		    xtick={0,5,10,15,20,25}, xmin=0, xmax=26]
		  \addplot[black, dashed, mark=*, mark size=1.6pt] 
		 coordinates {
		       (0,1.0000)
			(1,0.9988)
			(2,0.9888)
			(3,0.9875)
			(4,0.9744)
			(5,0.9630)
			(6,0.9067)
			(7,0.9047)
			(8,0.8974)
			(9,0.8698)
			(10,0.8667)
			(11,0.8614)
			(12,0.8564)
			(13,0.8146)
			(14,0.7825)
			(15,0.7065)
			(16,0.5026)
			(17,0.4147)
			(18,0.4043)
			(19,0.4018)
			(20,0.3695)
			(21,0.3553)
			(22,0.3102)
			(23,0.3102)
			(24,0.1813)
			(25,0.0492)
			(26,0.0438)
	       };
		\end{axis}
		\end{tikzpicture}
		
		\begin{tikzpicture}
			\begin{axis}[width=9cm, height=4.5cm, xlabel={generated RHS index}, ylabel={$\abs[auto]{\cos\prn{\bb^{(i)},\bb^{(NN)}}}$},
		    xtick={0,5,10,15,20,25}, xmin=0, xmax=26]
		  \addplot[black, dashed, mark=*, mark size=1.6pt] 
		  coordinates {
		  	(0,0.9988)
			(1,0.9988)
			(2,0.9888)
			(3,0.9967)
			(4,0.9967)
			(5,0.9844)
			(6,0.9895)
			(7,0.9964)
			(8,0.9964)
			(9,0.9987)
			(10,0.9964)
			(11,0.9964)
			(12,0.9987)
			(13,0.9925)
			(14,0.9867)
			(15,0.9867)
			(16,0.9922)
			(17,0.9963)
			(18,0.9963)
			(19,0.9914)
			(20,0.9979)
			(21,0.9988)
			(22,0.9988)
			(23,0.9979)
			(24,0.8956)
			(25,0.9985)
			(26,0.9985)
		};
		\end{axis}
		\end{tikzpicture}
		
	\end{center}
	\caption{
		 In the top plot, we show for each generated right-hand side
		 $\bb^{(i)}$ the absolute cosine of the angle between the vector and $\bb$, 
		noting that the angle bears seemingly little relation to which vectors produce 
		the same convergence curve.  The bottom plot shows the absolute cosine of each 
		$\bb^{(i)}$ with its nearest neighbor.  Though they are isolated, they 
		are not too far from one another on $\S^3$.
		\label{figure:cosines-prescribed-rhss}
	}
\end{figure}

For two of the generated right-hand sides (corresponding to rows 13 and 26 in
\Cref{figure:rhss-same-gmres-curve}), we show the vector as well as the
associated $W$ and $R$ from the APS-parameterization of the  
system generating the same convergence curve. 

From \Cref{figure:cosines-prescribed-rhss},
we observe that $\abs[auto]{\cos\prn{\bb,\bb^{(13)}}}=0.8146$, 
and we present the right-hand side, and quantities from the APS-parameterization,
\begin{align*}
	\bb^{(13)}
	=&
	\begin{bmatrix}
		 -0.0079 
		&
		 0.3236 
		&
		 0.7844 
		&
		 0.5291 
	\end{bmatrix}^T
	\quad
	\bw_1^{(13)}
	=
	\begin{bmatrix}
		 0.0830 
		&
		 0.3119 
		&
		 0.6251 
		&
		 0.7107 
	\end{bmatrix}^T
	\\
	W^{(13)}
	=&
	\begin{bmatrix}
		 0.0830 & 0.6576 & -0.5948 & -0.4550 
		\\
		 0.3119 & 0.5642 & 0.7557 & -0.1156 
		\\
		 0.6251 & 0.1665 & -0.2733 & 0.7119 
		\\
		 0.7107 & -0.4707 & -0.0218 & -0.5223 
	\end{bmatrix}
	\quad
	R^{(13)}
	=
	\begin{bmatrix}
		 2.3613 & 5.3834 & 12.5215 & 29.5418 
		\\
		 0 & 0.2591 & 1.0207 & 3.0787 
		\\
		 0 & 0 & 0.1274 & 0.4197 
		\\
		 0 & 0 & 0 & 0.1267 
	\end{bmatrix}
	.
\end{align*}
Similarly,
we observe that $\abs[auto]{\cos\prn{\bb,\bb^{(26)}}}=0.0438$, 
and we present the right-hand side, and quantities from the APS-parameterization,
\begin{align*}
	\bb^{(26)}
	=&
	\begin{bmatrix}
		 0.2752 
		&
		 -0.5273 
		&
		 0.6447 
		&
		 -0.4803 
	\end{bmatrix}^T
	\quad
	\bw_1^{(26)}
	=
	\begin{bmatrix}
		 -0.1092 
		&
		 0.4208 
		&
		 -0.6271 
		&
		 0.6463 
	\end{bmatrix}^T
	\\
	W^{(26)}
	=&
	\begin{bmatrix}
		 -0.1092 & 0.5592 & -0.6609 & 0.4884 
		\\
		 0.4208 & 0.6448 & 0.0045 & -0.6381 
		\\
		 -0.6271 & 0.4969 & 0.5926 & 0.0928 
		\\
		 0.6463 & 0.1569 & 0.4604 & 0.5880 
	\end{bmatrix}
	\quad
	R^{(26)}
	=
	\begin{bmatrix}
		 1.0015 & -0.7924 & 0.6699 & -0.4402 
		\\
		 0 & 0.1597 & 0.2710 & 0.7631 
		\\
		 0 & 0 & 0.1237 & 0.3838 
		\\
		 0 & 0 & 0 & 0.1267 
	\end{bmatrix}
	.
\end{align*}
We display these two right-hand sides as examples because they are, respectively, closer and further from $\bb$ on 
the unit sphere. 

\section{Conclusions and future work}
\label{section:conclusions}
	Our goal in this work was to understand how the general results of
	\cite{GreenbaumPtakStrakos:1996:1} and its manifestation as the
	APS-parameterization \cite{ArioliPtakStrakos:1998:1} could be leveraged
	to gain better understanding of \gmres performance for matrix
	structures appearing in application problems.  We have explored this
	for the non-symmetric tridiagonal Toeplitz matrix, the kind which
	appear in discretizations of one-dimensional \pdes.  We have shown
	precisely how the tridiagonal Toeplitz structure constrains the
	residual Arnoldi vectors and the APS-parameterization.  Furthermore,
	we have gone further by exploring the structure of the Jacobian of $\bF$ that
	maps the matrix and right-hand side to residual curves.  From this, we
	have been able to describe how neighborhoods of residual curves are
	attainable.  A small demonstration shows how this theory can be used to
	find right-hand sides yielding a prescribed \gmres convergence. 

	This demonstration hints at a much richer structure.  We kept the focus
	narrow, but it is clear that there is much further work that can be
	done by studying the underlying manifold structures.  In addition, the approach
	we have developed for the tridiagonal Toeplitz case can serve as a
	prototype for more complicated matrix structures, the kind which arise
	from \eg the discretization of more complicated \pdes.  This also includes 
	folding a deeper consideration of preconditioning into this theory, a 
	direction already pursued in \cite{MatalonSpillane:2025:1}.

	Furthermore, we have imposed no structure on the right-hand side.  The
	disconnection of the full-grade right-hand side space (\ie $\S^{n-1}$)
	leads us to ask what the theory tells us if we restrict our
	consideration to structured right-hand sides that would actually arise
	from discretization. 

	A more theoretical direction would be to more fully develop the tracing
	of level curve of $\bF$, \ie explore using $\bF'$ to construct families
	of pairs $(A,\bb)$ having the same convergence curve. For a specific,
	fixed convergence curve, the family of pairs $\prn{A,\bb}$ for which
	\gmres produce that convergence curve is nothing more than a level set
	of $\bF$.  Starting from one such pair, one can use $\bF'$ to move
	along the level set to obtain other tridiagonal
	matrix/right-hand side pairs producing the same convergence.

	\section*{Acknowledgments} The authors wish to thank Jen Pestana, PI
	for the EPSRC side of our joint project, for many insightful comments
	and suggestions. They also wish to thank Gerard Meurant for suggesting
	a deeper look at the APS parameterization based on some of their
	initial work. The first author thanks Jörg Liesen for providing a copy
	of \cite{LiesenMehrmann:2025:1}.  The first author also thanks Mark
	Embree for insightful and critical questions that led the authors to hone our
	results further. 

\appendix 
\section{The Frobenius normal form}\label{section:frobenius-normal-form}

We observe for the reader that this construction employed in, \eg,
\cite{ArioliPtakStrakos:1998:1,GreenbaumPtakStrakos:1996:1,GreenbaumStrakos:1994:1,
Meurant:2012:1,MeurantDuintjerTebbens:2014:1,Schweitzer:2016:1,MeurantDuintjerTebbens:2020:1},
can also be accomplished by careful specification of a matrix
via its Frobenius (\aka rational) canonical form.  This canonical form
represents $A$ as being similar to block diagonal matrix
\begin{align*}
	A=S^{-1}\bldiag\braces{C_{p_1},C_{p_2},\ldots,C_{p_k}}S
\end{align*}
wherein each diagonal block $C_{p_i}\in\C^{\deg p_i\times \deg p_i}$ is a
companion matrix.  

The APS-/GPS-/\gmres-convergence-specification factorization is not itself properly 
in Frobenius normal form, but it is related.
The construction is generated by running a Krylov power iteration for a
starting vector until an invariant subspace is reached. Thereafter, a new
starting vector in the orthogonal complement of the generated space is used to
continue the iteration, and this is repeated until $\C^n$ has been spanned. The
canonical form can be assembled by taking $S$ to have as its columns the
generated Krylov power bases, and the block diagonal of companion matrices
arises from \eqref{eqn:companion-matrix-relation} applied for each Krylov
basis. This is done in such a way that satisfies invariances that make it a 
\emph{canonical form}.
A nice treatment of this canonical form can be found in \cite[Chapter
16.3]{LiesenMehrmann:2025:1}.

One goal in this work is strongly related to the consideration of the constraints on
the Frobenius canonical form of a tridiagonal Toeplitz matrix.  However, we do
not approach our analysis with this directly in mind. 

\section{Proof of \Cref{lemma.resid-norm-deriv-a-tau}}\label{section:resid-norm-deriv-a-tau-PROOF}
We present the postponed proof of \Cref{lemma.resid-norm-deriv-a-tau}.
\begin{proof}
	We observe that for $\theta\in\braces{a,\tau, \xi}$, the entries of $\frac{\partial}{\partial\theta}P_k$ 
	and $\frac{\partial}{\partial\theta}Q_k$ all have the form 
	\begin{align}
		\frac{\partial}{\partial\theta} 
		\prn{A^j\bb}^TA^\ell\bb 
		= 
		\prn{\frac{\partial}{\partial\theta}A^j\bb}^T\prn{A^\ell\bb}
		+
		\prn{A^j\bb}^T\prn{\frac{\partial}{\partial\theta}A^\ell\bb}
		,
		\label{eqn:matrix-entry-deriv-gen}
	\end{align}
	since $\bb$ is independent of the parameters determining $A$. For $a$, it follows that  
	\begin{align}
		\Dparam{a} 
		\prn{A^j\bb}^TA^\ell\bb 
		= 
		\prn{j A^{j-1}\bb}^T\prn{A^\ell\bb}
		+
		\prn{A^j\bb}^T\prn{\ell A^{\ell-1}\bb}
		.
		\label{eqn:matrix-entry-deriv-a}
	\end{align}
	From the definition of the entries of $P_k$, we thus have the entry-wise derivative recurrence \wrt $a$
	\begin{align*} 
		\Dparam{a}
		\prn{P_k}_{j\ell} 
		=
		j\prn{P_k}_{j-1,\ell} 
		+
		\ell\prn{P_k}_{j,\ell-1}
		.
	\end{align*}
	We define the \emph{index-scaled downward shift matrix} 
	\begin{align*}
		D_k 
		= 
		\begin{bmatrix}
			 0& & && 
			\\ 
			   1 &0& & & 
			\\
			  & \ddots &0 &&  
			\\ 
			  && k-1 &\ddots &
			\\
			  &&& k &0 
		\end{bmatrix}
		,
	\end{align*}
	which allows us to express the derivative recursion at the matrix level as 
	\begin{align}
		\Dparam{a} P_k 
		= 
		D_k P_k + P_k D_k 
		. 
		\label{eqn:Pk-deriv-a}
	\end{align}
	For the entries of $Q_k$, the result is similar, but with a catch.  Observe that for the cases of either 
	$j=1$ or $\ell=1$, \eqref{eqn:matrix-entry-deriv-a} contains terms of the form $\bb^TA^\ell\bb$ or $\prn{A^j\bb}^T \bb$; 
	and in both cases the entry recursion produces quantities that are not entries of $Q_k$ (which does not arise for $P_k$).
	Thus, we must insert these quantities back into the derivative recursion as a rank-$2$ update, \ie 
	\begin{align}
		\Dparam{a} Q_k 
		= 
		D_kQ_k + Q_kD_k + \be_1\bp_k^T + \bp_k\be_1^T 
		\label{eqn:Qk-deriv-a}
		,
	\end{align}
	where $\bp_k = \prn{AK_k}^T\bb$.
	Due to $\trace\prn{D_k}=0$, and the linearity and cyclic invariance of the $\trace$, it follows that 
	\begin{align*}
		\trace\prn{P_k^{-1}\Dparam{a}P_k} 
		=& 
		\trace\prn{P_k^{-1}D_kP_k + D_k}
		= 
		0
		, 
		\quad 
		\mbox{and}
		\\
		\trace\prn{Q_k^{-1}\Dparam{a}Q_k} 
		=& 
		\trace\prn{Q_k^{-1}D_kQ_k + D_k + Q_k^{-1}\prn{\be_1\bp_k^T + \bp_k\be_1^T}}
		\\
		= &
		\trace\prn{Q_k^{-1}\prn{\be_1\bp_k^T + \bp_k\be_1^T}} = 2\be_1^T Q_k^{-1}\bp_k
	\end{align*}

	Similarly, for differentiation \wrt $\tau$, it follows from \eqref{eqn:matrix-entry-deriv-gen} and \eqref{eqn:Al-deriv-tau} that 
	\begin{align}
		\frac{\partial}{\partial\tau}
		\prn{A^j\bb}^TA^\ell\bb 
		=&
		\frac{j}{\tau} 
		\prn{
			\prn{A^j 
			- 
			a A^{j-1}}\bb
		}^T 
		A^\ell\bb 
		+ 
		\frac{\ell}{\tau}
		\prn{A^j\bb}^T
		\prn{
			\prn{A^\ell 
			- 
			a A^{\ell-1}}\bb
		}
		\nonumber
		\\
		=&
		\frac{j}{\tau}
		\brackets{
			\prn{A^j\bb}^TA^\ell\bb 
			- 
			a\prn{A^{j-1}\bb}^T A^\ell\bb
		}
		+ 
		\frac{\ell}{\tau}
		\brackets{
			\prn{A^j\bb}^T A^\ell\bb 
			- 
			a \prn{A^j\bb}^T A^{\ell-1}\bb
		}
		\label{eqn:matrix-entry-deriv-tau}
		. 
	\end{align}
	Observe that the second term being subtracted in each bracket mirrors a scaling of the term obtained 
	when differentiating \wrt $a$.  Let us define the \emph{tau- and index-scaled diagonal matrices} $N_k = \frac{1}{\tau}\diag\braces{1,\ldots, k}$ and $\widehat{N}_k= \frac{1}{\tau}\diag\braces{0,1,\ldots, k}$.  
	It follows from \eqref{eqn:matrix-entry-deriv-tau} that 
	\begin{align}
		\Dparam{\tau} P 
		= &
		\widehat{N}_k P_k + P_k \widehat{N}_k - a \Dparam{a}P_k
		,
		\quad 
		\mbox{and}
		\quad
		\Dparam{\tau} Q_k 
		= 
		N_k Q_k + Q_k N_k - a \Dparam{a}Q_k 
		\label{Pk-Qk-deriv-tau}
		.
	\end{align}
	Computing traces and again exploiting properties of $\trace$ as well as \eqref{eqn:Pk-deriv-a} and \eqref{eqn:Qk-deriv-a} yields 
	\begin{align*}
		\trace\prn{
			P_k^{-1}\prn{\widehat{N}_k P_k+P_k \widehat{N}_k -a\Dparam{a}P_k}
		}
		=& 
		2\trace\prn{N_k} 
		\\
		\trace\prn{
			Q_k^{-1}\prn{N_k Q_k+Q_k N_k -a\Dparam{a}Q_k}
		}
		=&
		2\trace\prn{N_k} 
		- 
		2a\be_1^T Q_k^{-1}\bp_k
		.
	\end{align*}
	The proof is completed by inserting these results into
	\eqref{eqn:log-based-residual-norm-deriv} for $a$ and $\tau$,
	simplifying, and observing that from the definitions of $Q_k$ and
	$\bp_k$, $Q_k^{-1}\bp_k$ is indeed the solution of
	\eqref{eqn:gmres-small-ls}. 
\end{proof}

\section{Proof of \Cref{lemma:resid-norm-deriv-xi}}\label{section:resid-norm-deriv-xi-PROOF}
We present the postponed proof of \Cref{lemma:resid-norm-deriv-xi}.
\begin{proof}
	The structure of this proof is similar to that of \Cref{lemma.resid-norm-deriv-a-tau}.  To differentiate 
	$P_k$ and $Q_k$ \wrt to $\xi$, we study the derivative of their entries; \ie 
	\begin{align}
		\Dparam{\xi}
		\prn{A^j \bb}^T 
		A^\ell \bb 
		=& 
		\prn{\Dparam{\xi}A^j \bb}^T 
		\prn{A^\ell \bb}
		+
		\prn{A^j \bb}^T 
		\prn{\Dparam{\xi}A^\ell \bb}
		\nonumber
		\\
		=&
		\frac{1}{\xi}
		\brackets{
			\prn{\brackets{N,A^j} \bb}^T 
			\prn{A^\ell \bb}
			+
			\prn{A^j \bb}^T 
			\prn{\brackets{N, A^\ell} \bb}
		}
		\nonumber
		\\
		=&
		\frac{1}{\xi}
		\brackets{
			\bb^T\prn{NA^j}^T A^\ell\bb 
			+ 
			\bb^T\prn{A^j}N A^\ell\bb 
			- 
			\bb^T\prn{A^j N}^TA^\ell\bb 
			- 
			\bb^T\prn{A^j}^T A^\ell N \bb
		}
		\label{eqn:matrix-entry-deriv-xi}
		,
	\end{align}
	where $N$ is defined as in \eqref{eqn:n-diag-indexing-matrix}. We denote
	\begin{math}
		\wtK_{i}
		\coloneq
		\begin{bmatrix}
			N\bb 
			& 
			AN\bb 
			& 
			\cdots 
			& 
			A^iN\bb
		\end{bmatrix}
	\end{math}
	as the Krylov basis of the \emph{index scaled} right-hand side, \ie $N\bb$. Then it 
	follows from \eqref{eqn:matrix-entry-deriv-xi} that we can compactly represent the matrix
	derivatives \wrt $\xi$ as 
	\begin{align*}
		\Dparam{\xi} 
		P_k 
		=&
		2K_{k+1}^T 
		N K_{k+1} 
		- 
		\wtK_{k+1}^T 
		K_{k+1} 
		- 
		K_{k+1}^T 
		\wtK_{k+1}
		\\
		\Dparam{\xi} 
		Q_k 
		=&
		2\prn{AK_k}^T 
		N A K_k 
		- 
		\prn{A\wtK_k}^T 
		\prn{AK_k} 
		- 
		\prn{AK_k}^T 
		\prn{A\wtK_k}
		.
	\end{align*}
	These quantities need to be inserted into the trace formulas from \eqref{eqn:log-based-residual-norm-deriv}; so we 
	analyze structure of these terms individually.  

	We observe that we can use trace linearity and cyclic invariance along with the pseudoinverse formula 
	\begin{math}
		K_{k+1}^\dagger 
		= 
		\prn{K_{k+1}^TK_{k+1}}^{-1}K_{k+1}^{T}
	\end{math}
	to simplify 
	\begin{align*}
		\trace\prn{P_k^{-1}\Dparam{\xi}P_k} 
		=& 
		\trace\prn{
			\prn{K_{k+1}^T K_{k+1}}^{-1} 
			\prn{
				2K_{k+1}^T N K_{k+1} 
				- 
				\wtK_{k+1}^T K_{k+1} 
				- 
				K_{k+1}^T\wtK_{k+1}
			}
		}
		\\ 
		=& 
		2\trace\prn{\Pi_{k+1} N} 
		- 
		2\trace\prn{\wtK_{k+1}K_{k+1}^\dagger}
		,
	\end{align*}
	where $\Pi_{k+1} \coloneq K_{k+1}\prn{K_{k+1}^T K_{k+1}}^{-1}K_{k+1}^T$ is the
	orthogonal projector onto $\cK_{k+1}(A,\bb)$.
	A similar computation with the other trace term yields 
	\begin{align*}
		\trace\prn{Q_k^{-1}\Dparam{\xi}Q_k} 
		=& 
		2\trace\prn{\Phi_k N}
		- 
		2\trace\prn{A\wtK_k\prn{AK_k}^T}
		,
	\end{align*}
	where $\Phi_k\coloneq AK_k\brackets{\prn{AK_k}^TAK_k}\prn{AK_k}^T$ is the orthogonal 
	projector onto $A\cK_k\prn{A,\bb}$.
	Inserting both trace formulas back into \eqref{eqn:log-based-residual-norm-deriv} for the derivative 
	\wrt $\xi$ yields, 
	\begin{align}
		\Dparam{\xi}
		\norm[auto]{\br_k} 
		= 
		\frac{\norm[auto]{\br_k}}{\xi}
		\brackets{
			\trace\prn{\Pi_{k+1}N}
			- 
			\trace\prn{\Phi_k N}
			- 
			\trace\prn{\wtK_{k+1}K_{k+1}^\dagger}
			+ 
			\trace\prn{A\wtK_k\prn{AK_k}^\dagger}
		}
		\label{eqn:log-based-residual-norm-deriv-xi}
	\end{align}
	We observe that by construction $A\cK_k(A,\bb)\subset \cK_{k+1}(A,\bb)$; and thus $\Pi_{k+1}-\Phi_k$ is a 
	projector onto the one-dimensional subspace of $\cK_{k+1}(A,\bb)$ that is not in $A\cK_k(A,\bb)$.  This is 
	precisely the \gmres residual $\br_k$; \ie $\Pi_{k+1}-\Phi_k = \br_k\br_k^T/\norm[auto]{\br_k}^2$. 
	It follows from this and from cyclic invariance $\trace$ that 
	\begin{align}
		\trace\prn{\Pi_{k+1}N}
		- 
		\trace\prn{\Phi_k N}
		=
		\trace\prn{\prn{\Pi_{k+1}-\Phi_k}N} 
		= 
		\frac{\br_k^T N\br_k}{\norm[auto]{\br_k}^2}
		\label{eqn:trace-component-1}
		.
	\end{align}

	To understand the structure of $\trace\prn{\wtK_{k+1}K_{k+1}^\dagger}$, we build on the work of \cite{Cline:1964:1,Greville:1960:1} 
	to understand the structure of 
	\begin{math}
		K_{k+1}^\dagger
		=
		\begin{bmatrix}
			\bb & AK_k
		\end{bmatrix}
		.
	\end{math}
	We begin with the ansatz 
	\begin{align*}
		K_{k+1}^\dagger 
		= 
		\begin{bmatrix}
			\bs^T 
			\\ 
			\prn{AK_k}^\dagger + \bu\bt^T
		\end{bmatrix}
	\end{align*}	
	and then solve for $\bs$, $\bt$, and $\bu$, using the fact that in this
	case we satisfy the pseudoinverse property $K_{k+1}^\dagger K_{k+1}=
	I$.  Expanding the product allows us to obtain the equations 
	\begin{align*}
		K_{k+1}^\dagger K_{k+1}
		= 
		\begin{bmatrix}
			\bs^T\bb & \bs^TAK_k 
			\\ 
			\brackets{\prn{AK_k}^\dagger + \bu\bt^T}\bb & \brackets{\prn{AK_k}^\dagger + \bu\bt^T}AK_k
		\end{bmatrix}
		= 
		I
		.
	\end{align*}
	From this, it immediately follows that $\bs=\br_k/\norm[auto]{\br_k}^2$ since from 
	\Cref{remark:AK-WR-column-relationship} and from the definition of the
	projector $\Phi_k$, it follows that $\bb = \Phi_k\bb + \br_k$.  This also yields automatically that $\bs^TAK_k=\bnull$. 
	Since $\prn{AK_k}^{\dagger}AK_k=I$ we make the ansatz $\bt=\alpha\br_k$ for some value of $\alpha$. It turns out 
	we actually do not need to solve for to complete the proof.  To obtain $\bu$, we solve 
	\begin{align*}
		\bnull
		=&
		\brackets{\prn{AK_k}^\dagger + \bu\bt^T}\bb 
		\\ 
		=&
		\prn{AK_k}^\dagger\bb
		+
		\underbrace{
			\bu\prn{\alpha \br_k}^T\Phi_k\bb 
		}_{
			=0 
			\atop 
			{
				\br_k \perp \range\prn{AK_k}
			}
		}
		+
		\underbrace{
			\bu\prn{\alpha \br_k}^T \br_k 
		}_{
			=\alpha\norm[auto]{\br_k}^2\bu
		}
		,
	\end{align*}
	and from this we conclude that $\alpha\bu = -\prn{AK_k}^\dagger\bb/\norm[auto]{\br_k}^2$.  
	Applying the cyclic invariance and linearity of the trace yields 
	\begin{align*}
		\trace\prn{\wtK_{k+1}K_{k+1}^\dagger} 
		=& 
		\trace\prn{K_{k+1}^\dagger \wtK_{k+1}}
		\\ 
		=& 
		\br_k^T N\bb/\norm[auto]{\br_k}^2 
		+ 
		\trace\prn{\prn{AK_k}^\dagger A\wtK_k} 
		- 
		\br_k^TA\wtK_k\prn{AK_k}^\dagger\bb/\norm[auto]{\br_k}^2
	\end{align*}
	It follows that 
	\begin{align*}
		\trace\prn{\wtK_{k+1}K_{k+1}^\dagger} 
		- 
		\trace\prn{A\wtK_k\prn{AK_k}^\dagger}
		= 
		\br_k^T\prn{N\bb-\prn{A\wtK_k}\prn{AK_k}^\dagger\bb}/\norm[auto]{\br_k}^2
		.
	\end{align*}

	We lastly make the observation that this relates back to a \gmres iteration and that 
	$\bc^{(k)}\coloneq\prn{AK_k}^\dagger\bb$ are by definition the coefficients of the \gmres iterate $\bx_k$ in the 
	Krylov power basis; \ie $\bx_k = K_k\bc^{(k)} = \sum_{i=1}^{k}c_i^{(k)}A^{i-1}\bb=\Theta_k(A)\bb$.  Thus, we can simplify 
	\begin{align*}
		\prn{A\wtK_k}\prn{AK_k}^\dagger\bb 
		= 
		\prn{A\wtK_k}\bc^{(k)}
		=& 
		\begin{bmatrix}
			AN\bb 
			& 
			A^2 N\bb 
			& 
			\cdots 
			& 
			A^k N\bb
		\end{bmatrix}
		\bc^{(k)}
		\\ 
		=& 
		\sum_{i=1}^k 
		\prn{c_i^{(k)}}A^iN\bb 
		=
		A\Theta_k(A)N\bb
		.
	\end{align*}
	This allows us to express the whole thing using the \gmres residual polynomial, proving the result. We have 
	\begin{align}
		\br_k^T\prn{N\bb-\prn{A\wtK_k}\prn{AK_k}^\dagger \bb}
		=& 
		\br_k^T\prn{N\bb-A\Theta_k(A)N\bb}
		\nonumber
		\\ 
		=& 
		\br_k^T\prn{I-A\Theta_k(A)}N\bb
		\nonumber
		\\ 
		=& 
		\br_k^T\Psi_k(A)N\bb
		\label{eqn:trace-component-2}
		.
	\end{align}
	Inserting \eqref{eqn:trace-component-1} and \eqref{eqn:trace-component-2} back into \eqref{eqn:log-based-residual-norm-deriv-xi} 
	yields the result, upon simplification to commutator form.
\end{proof}

\appendix

\printbibliography

\end{document}